\documentclass[reqno]{amsart}

\usepackage[utf8]{inputenc}
\usepackage[english]{babel}

\usepackage{soul}
\usepackage{braket}
\usepackage{amsmath}
\usepackage{amstext}
\usepackage{amssymb}
\usepackage{amsthm}
\usepackage{mathtools}
\usepackage{stmaryrd}
\usepackage{mathrsfs}
\usepackage{extarrows}
\usepackage{afterpage}
\usepackage{tabu}
\usepackage{array}
\usepackage{microtype}
\usepackage{graphicx}
\usepackage{amsfonts}
\usepackage{units}
\usepackage[shortlabels]{enumitem}
\usepackage{url}
\usepackage{bm}
\usepackage{tikz}
\usetikzlibrary{matrix,arrows,calc,fit,cd,positioning,intersections,arrows.meta,decorations.markings,decorations.pathmorphing,shapes,calligraphy}
\usepackage{tikz-3dplot}
\usepackage{subcaption}

\usepackage{hyperref}

\usepackage{xcolor}
\hypersetup{
    colorlinks,
    linkcolor={red!50!black},
    citecolor={blue!50!black},
    urlcolor={blue!80!black}
}

\usepackage[margin=1cm]{caption}
\usepackage{colortbl}
\usepackage{framed}
\usepackage{thmtools}
\usepackage{cleveref}

\allowdisplaybreaks[4]

\theoremstyle{definition}

\declaretheorem[name=Theorem, numberwithin=section]{theorem}
\declaretheorem[name=Lemma, sibling=theorem]{lemma}
\declaretheorem[name=Definition, sibling=theorem]{definition}
\declaretheorem[name=Remark, sibling=theorem]{remark}

\declaretheorem[name=Proposition, sibling=theorem]{proposition}
\declaretheorem[name=Corollary, sibling=theorem]{corollary}
\declaretheorem[name=Conjecture, sibling=theorem]{conjecture}

\newtheorem*{mainthm*}{Main Theorem}

\newcommand{\inter}{[1,w]}
\newcommand{\interpq}{[1,w_{p,q}]}
\newcommand{\refl}{\mathcal R}
\newcommand{\length}{\ell_{\mathcal R}}
\newcommand{\reflectionsb}{\mathcal{R}^{B}_{n}}
\newcommand{\reflectionsd}{\mathcal{R}^{D}_{n}}
\newcommand{\E}{\mathcal E}
\newcommand{\pq}{p,q}
\newcommand{\wpq}{w_{p,q}}
\newcommand{\uvr}{[u,v]}
\newcommand{\wpp}{w_{p}}

\newcommand{\kre}{K^{v}_{u}}
\newcommand{\ler}{\le}

\DeclareMathOperator{\MC}{MC}
\DeclareMathOperator{\NC}{NC}
\DeclareMathOperator{\NZ}{NZ}
\DeclareMathOperator{\ZZ}{Z}

\newcommand{\annularncB}{\NC^{B}}
\newcommand{\annularncD}{\NC^{D}}
\newcommand{\annularncBpq}{\NC^{B}(p,q)}
\newcommand{\annularncDpq}{\NC^{D}(p,q)}

\title[Noncrossing combinatorics of quasi-Coxeter elements]{Noncrossing combinatorics\\of quasi-Coxeter elements}
\author{Eugenio Cainelli \and Giovanni Paolini \and Riccardo Zaccaria}

\begin{document}

\begin{abstract}
       We study the combinatorics of the intervals below quasi-Coxeter elements in the absolute order of finite Coxeter groups.
    In particular, we prove that these intervals are EL-shellable, except possibly for quasi-Coxeter elements of types $E_7$ and $E_8$, and we compute their zeta polynomials.
    We also show that a total order of the reflections yields an EL-labeling if it does so on all rank-two subintervals, which proves a conjecture of Williams for Coxeter elements.
    In order to obtain these results, we prove them for the infinite families, and we treat the exceptional types by computer.
    In addition, we study the related interval monoids, and we find a counterexample to cancellativity for every proper quasi-Coxeter element of type $D_n$ ($n\ge5$), $E_6$, $E_7$, $E_8$; in particular, these monoids do not embed in the corresponding interval groups.
    This gives a negative answer to a question posed by Baumeister, Holt, Neaime, and Rees.

\end{abstract}

\maketitle

\section{Introduction}

Let $W$ be a Coxeter group and let $\refl$ be its set of reflections.
If $w \in W$ is a Coxeter element, then the interval $[1, w]$ in the absolute order is the famous noncrossing partition poset associated with $W$.
In type $A$, this recovers the classical lattice of noncrossing set partitions.
The combinatorics of $[1, w]$ has many connections with algebra and topology (e.g., quiver representation theory \cite{ingalls2009noncrossing}, braid groups and Artin groups \cite{Bessis.monoid}). Such connections are especially well-developed in the case of finite Coxeter groups, where $[1, w]$ is an EL-shellable lattice and its enumerative properties are controlled by Coxeter-Catalan combinatorics \cite{brady2005latticesfiniterealreflection, AthanasiadisBradyWatt2005,  armstrong2007generalizednoncrossingpartitionscombinatorics}.

The goal of this paper is to generalize the combinatorial study of $[1, w]$ to quasi-Coxeter elements $w$.
Recall that an element $w\in W$ is quasi-Coxeter if it admits a reduced reflection factorization
\[
w=t_1\cdots t_n,
\]
where \(n=\operatorname{rank}(W)\), such that the reflections \(t_1,\ldots,t_n\) generate \(W\).
Every Coxeter element is quasi-Coxeter, but the converse need not hold.
Quasi-Coxeter elements naturally arise in the study of the Hurwitz action on reflection factorizations; recently, interval groups arising from $[1,w]$ have been studied in relation to Artin groups \cite{baumeister2023interval-part1, baumeister2023interval-part2, baumeister2023isomorphism}.
Proper quasi-Coxeter elements (i.e., quasi-Coxeter elements that are not Coxeter elements) appear in type $D$ and in all exceptional types of rank at least three.

While the lattice property is known to fail in general, our first main result is
that EL-shellability holds for all quasi-Coxeter elements except possibly those of the exceptional types $E_7,E_8$:

\begin{theorem}
    Let $w$ be a quasi-Coxeter element of a finite irreducible Coxeter group $W$. If $W$ is not of type $E_7, E_8$, then the interval $[1,w]$ is EL-shellable.
    \label{thm:main-shellability}
\end{theorem}

We believe that the exceptions in \Cref{thm:main-shellability} are due to the limits of our computations in the exceptional cases.

\begin{conjecture}\label{conj:EL}
    The interval $[1,w]$ is EL-shellable for every quasi-Coxeter element $w$ of a finite irreducible Coxeter group.
\end{conjecture}

EL-shellability is a central property in poset topology.
Numerous posets associated with Coxeter groups are known to be EL-shellable, including: noncrossing partition posets of finite, affine, and rank-three Coxeter groups \cite{AthanasiadisBradyWatt2005, paolini2021proof, delucchi2024dual}; Bruhat intervals \cite{dyer1993hecke}; posets of layers of arrangements associated to root systems \cite{Delucchi_Girard_Paolini_2019, paolini2020dowling}.
The lattice property and EL-shellability of noncrossing partition posets have been crucial ingredients in the study of the $K(\pi, 1)$ property for spaces associated to Artin groups \cite{paolini2025dual}; for this reason, we expect \Cref{thm:main-shellability} to prove useful in the study of interval groups arising from quasi-Coxeter elements.

The most natural attempt at EL-shellability of $[1,w]$ consists in using the generating set $\refl$ itself as labels, with a suitable total order.
For this to work, the total order must satisfy a necessary condition called $w$-compatibility: every rank-two interval $[u,v] \subseteq [1, w]$ must have a unique rising chain.
In the case of Coxeter elements, Williams conjectured that the $w$-compatibility condition is also sufficient (see \cite[Section 7.3]{defant2025permutahedron}).
Refining \Cref{thm:main-shellability}, we prove the following:

\begin{theorem}
    Let $w$ be a quasi-Coxeter element of a finite irreducible Coxeter group $W$.
    Then every $w$-compatible total order of $\refl$ yields an EL-labeling of $[1,w]$.
    However, in types $H_4$, $E_7$, and $E_8$, there are quasi-Coxeter elements $w$ that admit no $w$-compatible total order, so that no total order of $\refl$ yields an EL-labeling of $[1,w]$.
    \label{thm:main-shellability-refined}
\end{theorem}

Williams's conjecture was proved by Thomas up to a finite check, but the proof has never been published.
During preparation of this manuscript, we communicated with Williams and Thomas; with permission, we include Thomas's proof in \Cref{rem:thomas.proof}.

On the enumerative side, recall that the zeta polynomial $Z_P(m)$ of a finite poset $P$ counts the number of multichains
\[
x_1\leq x_2\leq \cdots \leq x_{m-1}.
\]
In particular, $Z_P(2) = |P|$, and $Z_P(-1) = \mu_P(\hat 0, \hat 1)$ if $P$ is bounded with minimum element $\hat 0$ and maximum element $\hat 1$.
The zeta polynomials of finite noncrossing partition lattices have well-known formulas in terms of the degrees of the corresponding Coxeter groups.
We extend this classical result to intervals $[1,w]$ below quasi-Coxeter elements:

\begin{theorem}\label{thm:zeta.polynomial.D}
    If $W$ is of type $D_n$ and $w$ is a quasi-Coxeter element conjugate to $w_{p,q} = [1,\dots,p][p+1,\dots,n]$, then the zeta polynomial of $P = [1,w]$ is
    \[ Z_P(m) = \frac{p+q+2pq(m-1)}{(p+q)m} \binom{mp}{p} \binom{mq}{q}. \]
    The zeta polynomials for proper quasi-Coxeter elements in the exceptional types of rank at least three are recorded in \Cref{tab:exceptional-zeta.Weyl} and \Cref{tab:exceptional-zeta.H}.
    \label{thm:main-zeta}
\end{theorem}

Whenever $[1,w]$ is EL-shellable (in particular, if $W$ is not of type $E_7$ or $E_8$), \Cref{thm:main-shellability,thm:main-zeta} explicitly describe the homotopy type of the open interval $(1,w)$: it is a wedge of $(n-2)$-spheres, and the number of spheres is given by $|\mu_P(1, w)| = |Z_P(-1)|$.

Our last result concerns the interval monoid $M_w$, defined by the same presentation as the interval group $G_w$.
If $w$ is a Coxeter element, $M_w$ is the well-known dual Artin monoid \cite{Bessis.monoid}; since $[1,w]$ is a lattice, $M_w$ is a Garside monoid, so it is cancellative and embeds in $G_w$, which is isomorphic to the Artin group of $W$.
This is a key construction in the dual theory of Artin groups.
For a proper quasi-Coxeter element $w$ of a finite irreducible Coxeter group $W$ of rank at least three, the interval $[1,w]$ is a lattice if and only if $W$ is of type $H_3$, or $W$ is of type $H_4$ and $w$ has order 30 \cite[Theorem~2.12]{baumeister2023interval-part2}. Baumeister, Holt, Neaime, and Rees asked whether $M_w$ is still cancellative when the lattice property fails and whether it embeds in $G_w$ \cite[Section~4, Question~(d)]{baumeister2023interval-part2}.
We answer this question negatively for the types $D_n$ with $n \geq 5$, $E_6$, $E_7$, and $E_8$:

\begin{theorem}\label{thm:main-noncancellative}
    Let $w$ be a proper quasi-Coxeter element of a finite Coxeter group $W$ of type $D_n$ with $n \geq 5$, or of type $E_6$, $E_7$, or $E_8$.
    Then the interval monoid $M_w$ is neither left nor right cancellative.
    In particular, the natural map $M_w \to G_w$ is not injective.
\end{theorem}

The proof is based on a single counterexample.
For a proper quasi-Coxeter element $w$ of $D_5$, we find words $P$ and $Q$ of length six and a reflection $r$ such that $Pr = Qr$ in $M_w$, although $P \neq Q$ in $M_w$ (\Cref{prop:D5-counterexample}).
The two words were found by a computer search, but both properties can be checked by hand.
Since the interval monoid $M_x$ of any element $x \leq w$ embeds in $M_w$ (\Cref{prop:parabolic-submonoid}), the same failure of cancellativity occurs whenever $w$ lies above a proper quasi-Coxeter element of a parabolic subgroup of type $D_5$.
This is the case for every $w$ as in \Cref{thm:main-noncancellative}; in the exceptional types, we find such elements below $w$ by computer.
In type $D_4$, on the other hand, we have found that $M_w$ is cancellative and embeds in $G_w$, even if the interval below a proper quasi-Coxeter element of $D_4$ is not a lattice; we will present the proof in a forthcoming paper.
The question remains open for the proper quasi-Coxeter elements of types $F_4$ and $H_4$ whose interval $[1,w]$ is not a lattice (see \Cref{tab:exceptional-summary}).

The paper is organized as follows.
In \Cref{sec:preliminaries}, we recall the basic definitions and relevant prior results on Coxeter groups and posets.
In \Cref{sec:shellability-type-D}, we prove EL-shellability for intervals below quasi-Coxeter elements in the infinite families $A_n$, $B_n$, and $D_n$, and in \Cref{sec:zeta.polynomial}, we compute the zeta polynomial in type $D$.
\Cref{sec:exceptional-cases} presents the same results for the exceptional types.
Finally, in \Cref{sec:interval-monoid}, we prove \Cref{thm:main-noncancellative} on the interval monoids.

\section{Preliminaries}
\label{sec:preliminaries}

\subsection{Coxeter groups}
We recall some standard terminology on Coxeter groups; see \cite{Humphreys_1990,bjorner2005combinatorics} for a detailed exposition. A \emph{Coxeter system} is a pair \((W,S)\), where \(S\) is a finite set and
\[
W=\left\langle S \ \middle|\ s^{2}=1 \text{ for } s\in S,\ 
(st)^{m(s,t)}=1 \text{ for } s\neq t \right\rangle,
\]
with \(m(s,t)=m(t,s)\in \{2,3,4,\ldots\}\cup\{\infty\}\). The group \(W\) is called a
\emph{Coxeter group}, and the elements of \(S\) are its \emph{simple reflections}.
The \emph{rank} of \((W,S)\) is \(|S|\). 

We denote the set of all reflections by
\[
\refl=\{w s w^{-1}\mid w\in W,\ s\in S\}.
\]
For \(I\subseteq S\), the subgroup \(W_I=\langle I\rangle\) is called a
\emph{standard parabolic subgroup}; a \emph{parabolic subgroup} of \(W\) is any
conjugate of such a subgroup. A \emph{Coxeter element} of \(W\) is a conjugate of the product of
the elements of \(S\), each appearing exactly once, in some order. When \(W\) is
finite, all Coxeter elements of \(W\) are conjugate. More generally,
a Coxeter element of a parabolic subgroup is called a \emph{parabolic Coxeter
element}. If \(w\in W\), its \emph{reflection length} is
\[
\length(w)=\min\{k\mid w=t_1\cdots t_k,\ t_i\in \refl\}.
\]
A factorization \(w=t_1\cdots t_k\) with \(t_i\in\refl\) and
\(k=\ell_\refl(w)\) is called a \emph{reduced reflection factorization}. An element \(w\in W\) is called a \emph{quasi-Coxeter element} if it admits a
reduced reflection factorization
\[
w=t_1\cdots t_n, \qquad t_i\in\refl,\quad n=|S|,
\]
such that the reflections \(t_1,\ldots,t_n\) generate \(W\). Similarly, \(w\) is
a \emph{parabolic quasi-Coxeter element} if it admits a reduced reflection
factorization whose factors generate a parabolic subgroup of \(W\).

\begin{remark}\label{rem:quasi-coxeter-not-conjugate}
Unlike Coxeter elements, quasi-Coxeter elements of a finite Coxeter group do
not, in general, form a single conjugacy class. Following Carter's
classification \cite{carter1972conjugacy}, the conjugacy classes of proper
quasi-Coxeter elements are described by Carter diagrams.
\end{remark}

\subsection{Posets}
We recall some elementary terminology on partially ordered sets; see \cite{Stanley_2011} for more details. A \emph{poset}
is a set \(P\) together with a partial order \(\leq\). If \(p<q\) and there is
no element \(r\in P\) such that \(p<r<q\), then we say that \(q\) \emph{covers}
\(p\), and write \(p\lessdot q\). The \emph{Hasse diagram} of \(P\) is the graph
whose vertices are the elements of \(P\), with an edge \(p\to q\) whenever
\(p\lessdot q\). For \(p\leq q\), the \emph{interval} between \(p\) and \(q\) is $[p,q]=\{r\in P\mid p\leq r\leq q\}.$
A poset is \emph{bounded} if it has a least element and a greatest
element. A \emph{chain} is a totally ordered subset
\(p_0<p_1<\cdots <p_k\), and its length is \(k\). The poset \(P\) is
\emph{graded} if it admits a strictly order-preserving function
$\operatorname{rk}\colon P\to\mathbb Z$, called a \emph{rank function}, such that
$\operatorname{rk}(q)=\operatorname{rk}(p)+1$ whenever $p\lessdot q$.
Every maximal chain in an interval $[p,q]$ then has length
$\operatorname{rk}(q)-\operatorname{rk}(p)$. Finally, \(P\) is a \emph{lattice} if every pair of
elements \(p,q\in P\) has a greatest common lower bound, called their
\emph{meet}, and a least common upper bound, called their \emph{join}.

\subsection{Intervals in groups, interval groups, and interval monoids}
\label{subsec:interval-groups}
Let \(G\) be a group generated by a symmetric set \(R=R^{-1}\). We denote by
\(\ell_R(g)\) the word length of \(g\) with respect to \(R\). The associated
\emph{distance order} on \(G\) is defined by
\[
x\leq y
\quad\Longleftrightarrow\quad
\ell_R(x)+\ell_R(x^{-1}y)=\ell_R(y).
\]
Equivalently, \(x\leq y\) if some shortest \(R\)-factorization of \(y\) begins
with a shortest \(R\)-factorization of \(x\). For \(g\in G\), we write $[1,g]=\{x\in G\mid x\leq g\}$ for the interval from the identity to \(g\). We also write $\operatorname{rank}(x) = \ell_R(x)$, which is the rank of the poset $[1,x]$.

We label every cover relation $x\lessdot y$ in $[1,g]$ by $x^{-1}y\in R$.
The \emph{interval group} associated with \(g\), denoted \(G_g\), is defined from
this labeled Hasse diagram. Let
\[
R_g=\{r\in R\mid \text{the label } r \text{ occurs on an edge of } [1,g]\}.
\]
Then \(G_g\) is the group generated by \(R_g\), with one relation for every pair
of directed paths in the Hasse diagram of \([1,g]\) having the same endpoints:
if the two paths have labels \(r_1,\ldots,r_k\) and \(s_1,\ldots,s_m\), then we
impose $r_1\cdots r_k=s_1\cdots s_m.$ Thus the presentation of \(G_g\) records exactly the relations visible inside
the interval \([1,g]\).
The \emph{interval monoid} \(M_g\) is the monoid defined by the same presentation.
The identity map of \(R_g\) induces a natural homomorphism \(M_g \to G_g\), which need not be injective.

In the case of a Coxeter group \(W\), we will take \(R=\refl\), the
set of reflections. The corresponding order is called the \emph{absolute order}. For \(w\in W\), the interval \([1,w]\) is the main object studied
below.
If \(w\) is a quasi-Coxeter element, then every reflection lies below \(w\), so that \(G_w\) and \(M_w\) are generated by \(\refl\).
Moreover, the Hurwitz action is then transitive on the reduced reflection factorizations of every element of \([1,w]\) \cite[Corollary~6.11 and Theorem~1.1]{baumeister2015hurwitzactionfinitecoxeter}, so that relations of length two suffice.

\begin{lemma}\label{lem:rank-two-presentation}
    Let \(w\) be a quasi-Coxeter element of a finite Coxeter group \(W\).
    Then the monoid \(M_w\) is presented by the generating set \(\refl\) and the relations \(ab = a'b'\), for all reflections \(a,b,a',b'\) such that \(ab = a'b'\) is an element of rank two below \(w\).
    Viewed as a group presentation, the same presentation defines \(G_w\).
\end{lemma}

For the group \(G_w\), this is \cite[Proposition~2.15]{baumeister2023interval-part2}, and the monoid \(M_w\) is the one considered in \cite[Section~4, Question~(d)]{baumeister2023interval-part2}.

\begin{theorem}[{\cite{brady2005latticesfiniterealreflection}}]\label{thm:coxeter-interval-lattice}
Let \(W\) be a finite Coxeter group and let \(w\in W\) be a Coxeter element.
Then the interval $[1,w]$ in the absolute order is a lattice.
\end{theorem}

The intervals \([1,w]\) are the Coxeter-theoretic noncrossing
partition lattices; see \cite{armstrong2007generalizednoncrossingpartitionscombinatorics} for an overview.
Whenever \([1,w]\) is a lattice, the monoid \(M_w\) is a Garside monoid with Garside element \(w\) \cite[Theorem~0.5.2]{Bessis.monoid} (see also \cite[Chapter~VI]{DDGKM}); in particular, \(M_w\) is cancellative and embeds in \(G_w\).
For a Coxeter element \(w\), this gives the dual Artin monoid of \(W\), and \(G_w\) is isomorphic to the Artin group of \(W\) \cite{Bessis.monoid}.

\begin{remark}\label{rem:quasi-coxeter-intervals-not-always-lattices}
The previous theorem does not extend to all quasi-Coxeter elements.
For example, if \(w\) is a proper quasi-Coxeter element
in a Coxeter group of type \(D_n\), with \(n\geq 4\), the interval $\inter$ is not a lattice. This result is proven in \cite[Section 5]{NICA.annular.noncrossing} for annular noncrossing partition posets. Thus intervals below quasi-Coxeter elements form a natural
generalization of noncrossing partition intervals, but the lattice property
does not necessarily hold in this setting.
\end{remark}

Without the lattice property, the monoid \(M_w\) can fail to be cancellative (\Cref{thm:main-noncancellative}).
Finally, we recall how intervals behave with respect to parabolic subgroups.
The \emph{parabolic closure} \(W_x\) of an element \(x\in W\) is the smallest parabolic subgroup of \(W\) that contains \(x\).
The following lemma is a consequence of a theorem of Brady and Watt \cite[Theorem~1]{BradyWatt2002orthogonal}; for quasi-Coxeter elements, it is \cite[Proposition~2.8]{baumeister2023interval-part2}.

\begin{lemma}\label{lem:mov-embedding}
    Let \(w\in W\) and \(x\le w\).
    Then \([1,w]\cap W_x=[1,x]\).
\end{lemma}

Hence every defining relation of \(M_w\) that applies to a word in the reflections below \(x\) is a defining relation of \(M_x\), which gives the following.

\begin{proposition}\label{prop:parabolic-submonoid}
    Let \(w\in W\) and \(x\le w\).
    Then the inclusion of generating sets induces an injective homomorphism \(M_x\to M_w\).
    Consequently, if \(M_x\) is not left (respectively, right) cancellative, then neither is \(M_w\).
\end{proposition}

\subsection{Signed permutations and quasi-Coxeter elements}
In this paper, we study intervals below quasi-Coxeter elements in type $D$ and in the exceptional types. Among finite irreducible Coxeter groups of rank at least three, proper quasi-Coxeter elements occur precisely in types $D_n$ ($n\ge4$), $E_6$, $E_7$, $E_8$, $F_4$, $H_3$, and $H_4$; see \cite{baumeister2015hurwitzactionfinitecoxeter,baumeister2023interval-part2}. In type $I_2(m)$ with simple generators $s,t$, the quasi-Coxeter elements are the rotations $(st)^k$ with $\gcd(k,m)=1$; such an element is proper exactly when $k\not\equiv\pm1\pmod m$.

We introduce notation to describe the Coxeter groups of type $B$ and $D$ as subgroups of the symmetric group on the set
$[\pm n]\coloneqq \{1, \dots, n,-1, \dots, -n\}$.
For $i_1,\ldots,i_k\in[\pm n]$ with pairwise distinct absolute values, define:
\begin{equation*}
\begin{split}
    ((i_{1}, i_{2}, \dots , i_{k})) &\coloneqq 
(i_{1}, i_{2}, \dots, i_{k})(-i_{1},-i_{2}, \dots ,-i_{k}); \\ 
[i_{1}, i_{2}, \dots , i_{k}] &\coloneqq
(i_{1}, i_{2}, \dots , i_{k},-i_{1},-i_{2}, \dots ,-i_{k}).
\end{split}
\end{equation*}
Throughout the paper, we compose permutations from right to left, so that $(uv)(a) = u(v(a))$; for example, $(a,b)(b,c) = (a,b,c)$.

Using the previous notation, the set of reflections of $B_{n}$ is
\begin{align*}
  \reflectionsb \coloneqq
\{((i, j)) \mid 1 \le i < j \le n\}
\;&\cup\;
\{((i,-j)) \mid 1 \le i < j \le n\} \\
\;\cup\; 
\{(i,-i) \mid &1 \le i \le n\};  
\end{align*}
and the set of reflections of $D_{n}$ is
\[\reflectionsd \coloneqq
\{((i, j)) \mid 1 \le i < j \le n\}
\;\cup\;
\{((i,-j)) \mid 1 \le i < j \le n\}.
\]

Let $p$ and $q$ be two positive integers and let $n=p+q$. The elements
$w_{p,q} \coloneqq [1, 2, \dots , p][p+1, \dots ,n] \in D_{n}$
are quasi-Coxeter. As $p$ and $q$ vary, they represent all the possible conjugacy classes of quasi-Coxeter elements in $D_{n}$. These classes are characterized by Carter diagrams in \cite{carter1972conjugacy}.

\subsection{Annular noncrossing partitions}
\label{sec:annular}
Since we aim to study intervals in finite Coxeter groups of type $D$ under quasi-Coxeter elements, we review the combinatorial description of posets of this kind in terms of annular noncrossing partitions, as described by Nica and Oancea \cite{NICA.annular.noncrossing}.
We first describe the annular noncrossing partitions of type $B$, and then of type $D$.

Let $p+q=n$ as before. A partition $\pi$ of $[\pm n]$ is invariant under sign change if $-A\coloneqq\{-a\mid a\in A\}$ is a block of $\pi$ whenever $A$ is a block of $\pi$.
For such a partition, a block $A$ satisfying $A=-A$ is called an inversion-invariant block, or a zero-block. All other blocks are called nonzero blocks and occur in distinct pairs $A,-A$, with $A\cap(-A)=\emptyset$.
We define the set $\annularncB(p,q)$ of annular noncrossing partitions of type $B$ to consist of the partitions of $[\pm n]$ that are invariant under sign change, have at most one zero-block, and can be drawn without crossing inside an annulus which has the points $1, \dots, p,-1,\dots,-p$ marked clockwise on its outer circle, and has the points $p+1, \dots, n, -(p+1), \dots, -n$ marked counterclockwise on its inner circle.

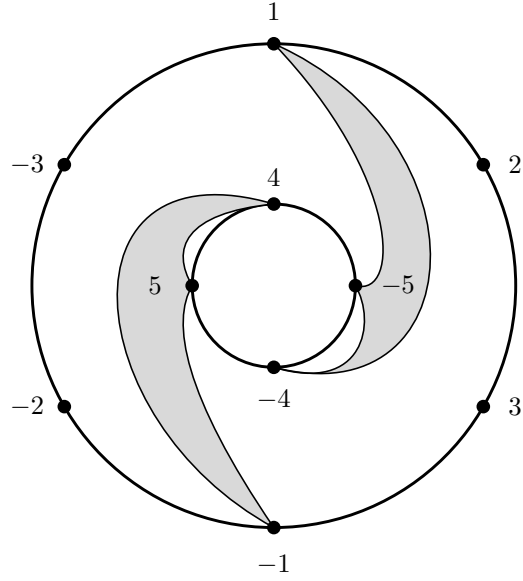
\begin{figure}[ht]
\centering
\begin{tikzpicture}[scale=0.8]

\def\R{4}
\def\r{1.35}

\draw[line width=1.1pt] (0,0) circle (\R);
\draw[line width=1.1pt] (0,0) circle (\r);

\coordinate (P1)  at (90:\R);
\coordinate (P2)  at (30:\R);
\coordinate (P3)  at (-30:\R);
\coordinate (Pm1) at (-90:\R);
\coordinate (Pm2) at (-150:\R);
\coordinate (Pm3) at (150:\R);

\coordinate (P4)  at (90:\r);
\coordinate (Pm5) at (0:\r);
\coordinate (Pm4) at (-90:\r);
\coordinate (P5)  at (180:\r);

\fill[gray!30]
(P1)  .. controls (3.6,2.4)  and (3.3,-2.2)  .. (Pm4)
      .. controls (1.5,-1.8) and (1.7,-0.6)  .. (Pm5)
      .. controls (2.1,-0.2) and (2.1,1.8)   .. (P1)
-- cycle;

\fill[gray!30]
(P4)  .. controls (-3.6,2.4) and (-3.3,-2.2) .. (Pm1)
      .. controls (-1.5,-1.8) and (-1.7,-0.6) .. (P5)
      .. controls (-1.7,0.6)  and (-1.5,1.2)  .. (P4)
-- cycle;

\draw[black, semithick] (P1) .. controls (3.6,2.4) and (3.3,-2.2) .. (Pm4);

\draw[black, semithick] (Pm4) .. controls (1.5,-1.8) and (1.7,-0.6) .. (Pm5);

\draw[black, semithick] (P1) .. controls (2.1,1.8) and (2.1,-0.2) .. (Pm5);

\draw[black, semithick] (P4) .. controls (-3.6,2.4) and (-3.3,-2.2) .. (Pm1);

\draw[black, semithick] (Pm1) .. controls (-1.5,-1.8) and (-1.7,-0.6) .. (P5);

\draw[black, semithick] (P4) .. controls (-1.5,1.2) and (-1.7,0.6) .. (P5);

\foreach \x in {P1,P2,P3,Pm1,Pm2,Pm3,P4,Pm5,Pm4,P5}
{
    \fill (\x) circle (3.2pt);
}

\node[yshift=12pt] at (P1) {$1$};
\node[xshift=12pt] at (P2) {$2$};
\node[xshift=12pt] at (P3) {$3$};

\node[yshift=-14pt] at (Pm1) {$-1$};
\node[xshift=-14pt] at (Pm2) {$-2$};
\node[xshift=-14pt] at (Pm3) {$-3$};

\node[yshift=10pt] at (P4) {$4$};
\node[xshift=16pt] at (Pm5) {$-5$};
\node[yshift=-12pt] at (Pm4) {$-4$};
\node[xshift=-14pt] at (P5) {$5$};

\end{tikzpicture}
\caption{The annular noncrossing partition of {$((1,-4,-5))$}.}
\label{fig:annular.example2}
\end{figure}

For example, the partition $\pi'=\{\{1,-4,-5\},\{-1,4,5\},\{2\}, \{3\}, \{-2\}, \{-3\}\}$ is an element of $\annularncB(3,2)$, since it can be drawn without crossing blocks in the corresponding annulus, as shown in Figure \ref{fig:annular.example2}.

We consider the poset structure on $\annularncB(p,q)$ given by the reverse refinement order
(\(
\pi \leq \rho\) if and only if \(
\text{every block of } \rho \text{ is a union of blocks of } \pi
\)).
By \cite[Theorem 1.4]{NICA.annular.noncrossing}, there exists a poset isomorphism between $\annularncB(p,q)$ and the type $B$ interval $\interpq_{B}$.
More precisely, the isomorphism assigns to every element $\tau \in \interpq_{B}$ the partition obtained from the sets of entries in the disjoint cycles of $\tau$ by retaining all nonzero blocks and merging all zero-blocks, if any, into a single zero-block.
For example, the noncrossing partition corresponding to $((1,-4,-5)) \in [1,w_{3,2}]_{B}$ is $\pi'$ (see again Figure \ref{fig:annular.example2}).
Since reflection length in $D_n$ agrees with reflection length in $B_n$ on elements of $D_n$, the interval $\interpq_{D}$ is an induced subposet of $\interpq_{B}$. We can therefore restrict the preceding poset isomorphism to $\interpq_{D}$, and we denote its image by $\annularncD(p,q)$. The elements of $\annularncD(p,q)$ are the annular noncrossing partitions of type $D$. An element of $\annularncB(p,q)$ lies in $\annularncD(p,q)$ if and only if either it has no zero-block, or its zero-block contains elements of both the inner and the outer circle.
Indeed, $D_n$ consists of the elements of $B_n$ that are even as permutations of $[\pm n]$ \cite[Section 7]{NICA.annular.noncrossing}. A pair of cycles $((i_1,\dots,i_r))$ is an even permutation, whereas a cycle $[i_1,\dots,i_r]$ is odd. By \cite[Lemma 4.8]{NICA.annular.noncrossing}, a zero-block meeting both circles comes from exactly two cycles of the latter form, one on each circle, and any other zero-block comes from exactly one.
In the following, we identify the interval $\interpq_{B}$ with $\annularncB(p,q)$, and $\interpq_{D}$ with $\annularncD(p,q)$.

In type $D$, a poset of annular noncrossing partitions is isomorphic to a corresponding interval below the quasi-Coxeter element $w_{\pq}$.
In type $B$, the maximal element does not correspond to a quasi-Coxeter element of $B_{n}$.
However, it is useful to understand the noncrossing structure in both cases since the zeta polynomial is already calculated in type $B$ (see \cite[Section 5]{GouldenNicaEnumerativeProperties}).

\subsection{Shellability}
We say that a simplicial complex is pure if all its maximal simplices have the same dimension. We first recall the definition of shellability for a pure simplicial complex.
\begin{definition}\label{def:shellability}
    A finite pure simplicial complex $X$ is shellable if there exists a total order (called a shelling order) on its maximal simplices
    $\sigma_{1},\sigma_{2},\ldots,\sigma_{m}$ such that
    $\sigma_{i} \cap \left(\bigcup_{j<i} \sigma_{j}\right)$
    is pure of dimension $\dim (\sigma_{i})-1$ for every $2 \leq i \leq m$.
\end{definition}
In the rest of this subsection, we denote by $P$ a finite bounded graded poset. We write $\operatorname{rank}(P)$ for the common length of its maximal chains.
The order complex $\Delta(P)$ is the simplicial complex whose vertices are the elements of $P$ and whose faces are the finite chains of $P$. We say that $P$ is shellable if $\Delta(P)$ is shellable.
Let $\Lambda$ be a totally ordered set of labels. We denote by $\E(P)$ the set of cover relations of $P$. An edge labeling for $P$ is a map $\lambda: \mathcal{E}(P) \to \Lambda$.
For every saturated chain $c:x_0\lessdot x_1\lessdot\cdots\lessdot x_k$ in $P$, we set
\[
\lambda(c)=(\lambda(x_0,x_1),\ldots,\lambda(x_{k-1},x_k)).
\]
The chain $c$ is said to be rising (or falling) with respect to $\lambda$ if the entries of $\lambda(c)$ are strictly increasing (or weakly decreasing, respectively) with respect to the total order of $\Lambda$.
For saturated chains $c$ and $c'$ of the same length in $P$, we say that $c$ is lexicographically smaller than $c'$ (with respect to $\lambda$) if
$\lambda(c)$ precedes $\lambda(c')$ in the lexicographic order induced
by the total order of $\Lambda$.

\begin{definition}\label{def:EL.labeling}
An edge labeling $\lambda: \mathcal{E}(P) \to \Lambda$ is called an \emph{EL-labeling} if for every non-singleton interval $[u,v]$ in $P$:
\begin{enumerate}[label=(\roman*)]
    \item there is a unique rising maximal chain in $[u, v]$;
    \item this chain is lexicographically smallest among all maximal chains in $[u, v]$, with respect to $\lambda$.
\end{enumerate}
\end{definition}

We say that $P$ is EL-shellable if there exists an EL-labeling $\lambda: \E(P) \to \Lambda$ for $P$.
Every EL-shellable poset is shellable, as stated in the next theorem, proved by Bj{\"o}rner \cite[Theorem 2.3]{Bjorner1980ShellableAC}.
\begin{theorem}\label{thm:EL.shellability}
Let $\lambda: \mathcal{E}(P) \longrightarrow \Lambda$ be an EL-labeling for $P$. Then any ordering $c_{1},c_{2},\ldots,c_{m}$ of the maximal chains of $P$ that is lexicographic with respect to their label sequences, with ties broken arbitrarily, induces a shelling order for $\Delta(P)$.
\end{theorem}
We denote by $\hat{0}$ the minimum and by $\hat{1}$ the maximum of $P$. Theorem \ref{thm:EL.shellability} gives precise information about the combinatorics and the topology of $\hat{P}= P \setminus \{\hat{0}, \hat{1}\}$.  
Let $\hat{c_{i}}=c_{i} \setminus \{\hat{0}, \hat{1}\}$ for every $1 \leq i \leq m$. 
The order $\hat{c_{1}},\hat{c_{2}},\dots,\hat{c_{m}}$ induces a shelling order for $\Delta(\hat{P})$.
Also, the simplicial complex $\Delta(\hat{P})$ is homotopy equivalent to a wedge of $(\operatorname{rank}(P)-2)$-dimensional spheres, where the number of spheres is equal to the number of falling maximal chains of $P$ with respect to any EL-labeling. Hence, studying the EL-labelings of a poset $P$ also gives precise information about the homology of the order complex $\Delta(\hat{P})$.

We now define a particular edge labeling for intervals in finite Coxeter groups, the natural edge labeling.
Let \(\mathcal{R}\) be the set of reflections of a finite Coxeter group \(W\), and let $w \in W$ be a quasi-Coxeter element.
We define the natural edge labeling $\lambda$ of \(\inter\) as follows:
\begin{equation}\label{eq:natural.edge.labeling}
 \lambda(u,v)=u^{-1}v \in \mathcal{R}
\end{equation}
for every cover relation \(u \lessdot v\) in \(\inter\).

\begin{theorem}[{\cite[Theorem 4.2]{AthanasiadisBradyWatt2005}}]
    \label{thm:EL-labeling.for.Coxeter.elements}
    Let \(W\) be a finite Coxeter group and let \(w \in W\) be a Coxeter element. Then there exists a total order on the set of 
    reflections $\mathcal{R}$ such that the natural edge labeling $\lambda:\E(\inter)\to \mathcal{R}$ of \(\inter\) is an EL-labeling.
\end{theorem}

To prove EL-shellability of intervals, it is convenient to recall the following two standard facts, which also apply to intervals below quasi-Coxeter elements $w \in W$.

\begin{remark}
     Let $\lambda$ be the natural edge labeling of $[1, w]$, where the set of labels $\refl$ is totally ordered.
     Then every interval $[u,v] \subseteq [1,w]$ has a unique lexicographically smallest maximal chain, and this chain is rising.
     The proof is the same as \cite[Theorem 3.5 (i)]{AthanasiadisBradyWatt2005} and relies on $\refl$ being closed under conjugation.
    \label{rmk:rising.chain}
\end{remark}

In the previous remark, one can analogously observe that every interval $[u, v] \subseteq [1,w]$ has a unique colexicographically largest maximal chain, and this chain is rising.
Here, the colexicographic order is simply the lexicographic order applied to the reversed sequence of edge labels; that is, two maximal chains are compared by looking first at their final labels, then at their penultimate labels, and so on.

\begin{remark}
    Let $\uvr$ be an interval in $[1, w]$ and let $u' = u^{-1}v$. Then, left-multiplication by $u$ is a poset isomorphism 
    $f:[1,u'] \to \uvr$ such that $\lambda(x,y)=\lambda(f(x),f(y))$ for all cover relations $(x,y) \in \E([1, u'])$.
    The proof is the same as \cite[Lemma 3.7]{AthanasiadisBradyWatt2005}.
    \label{rmk:interval.isomorphism}
\end{remark}

We say that a total order of the set of reflections $\refl$ is $w$-compatible if every rank-two interval $[1, u] \subseteq [1, w]$ has a unique rising maximal chain. This is clearly a necessary condition for the natural edge labeling to be an EL-labeling of the entire interval $[1,w]$.
Note that this is a much weaker condition than the $w$-compatibility of \cite{AthanasiadisBradyWatt2005}.

In the case of Coxeter elements, Williams conjectured that the $w$-compatibility condition is also sufficient to obtain an EL-labeling; more generally, Mühle conjectured that the same holds for Coxeter elements of well-generated complex reflection groups \cite[Conjecture 4.4]{muhle2015shellability}.
By \cite[Theorem 5.23]{muhle2020connectivity}, for intervals in finite Coxeter groups, the natural edge labeling of $[1,w]$ is an EL-labeling with respect to a fixed total order if and only if that order is $w$-compatible and $[1,w]$ is totally well-covered with respect to it, in the sense of \cite[Definition 5.17]{muhle2020connectivity}.

\section{EL-shellability in the infinite families}\label{sec:shellability-type-D}

This section is dedicated to proving \Cref{thm:main-shellability,thm:main-shellability-refined} for the infinite families $A_n$, $B_n$, and~$D_n$.
First, we introduce some terminology and two auxiliary results.
In all of them, $w$ is an element of a finite Coxeter group $W$ and $\prec$ is a $w$-compatible total order on $\refl$.

\begin{definition}
    \label{def:terminal}
    A reflection $r \leq w$ is \emph{$w$-terminal} (with respect to $\prec$) if $t \prec r$ for every reflection $t \leq wr$.
\end{definition}

If the interval $[1,wr]$ has a unique rising maximal chain, then this chain is the colexicographically largest one (see \Cref{rmk:rising.chain} and the subsequent discussion), so it ends with the $\prec$-largest reflection below $wr$.
Therefore, appending the cover relation $wr \lessdot w$, which is labeled $r$, gives a rising maximal chain of $[1,w]$ if and only if $r$ is $w$-terminal.

\begin{remark}[Rank-two intervals]
    \label{rmk:rank-two}
    Let $u \in [1,w]$ be an element of rank $2$.
    The maximal chains of $[1,u]$ are $1 \lessdot r \lessdot u$, labeled $(r, ru)$, where $r$ ranges over the reflections below $u$.
    The map $r \mapsto ru$ permutes the reflections below $u$ without fixed points, so each of its cycles contains some $r$ with $r \prec ru$.
    Since $\prec$ is $w$-compatible, $[1,u]$ has a unique rising maximal chain, and therefore this permutation is a single cycle.
    Hence, the reflections below $u$ can be cyclically ordered $r_1, \dots, r_m$ (with indices modulo $m$) so that $u = r_h r_{h+1}$ for every $h$, with a unique index $h$ such that $r_h \prec r_{h+1}$.
\end{remark}

The key step in the proof of EL-shellability is to show that two distinct reflections cannot both be $w$-terminal.
The following lemma rules this out for a recurring configuration of reflections, which arises in each of types $A$, $B$, and $D$.
\Cref{fig:typeA-crossing-reflections} shows this configuration in type $A$: two crossing reflections $r$ and $r'$ are the diagonals of a quadrilateral with sides $a_1, \dots, a_4$, and the elements $a_1a_2$, $a_2a_3$, $a_3a_4$, $a_4a_1$ correspond to the four triangles formed by two consecutive sides and a diagonal.

\begin{figure}[ht]
\centering
\begin{tikzpicture}[scale=0.8]

\def\R{2}

\draw[line width=1.1pt] (0,0) circle (\R);

\coordinate (Pi) at (45:\R);
\coordinate (Pk) at (-15:\R);
\coordinate (Pj) at (-90:\R);
\coordinate (Pl) at (-135:\R);

\draw[black, semithick] (Pi) -- (Pj);
\draw[black, semithick] (Pk) -- (Pl);

\draw[semithick, dashed] (Pi) -- (Pk);
\draw[semithick, dashed] (Pk) -- (Pj);
\draw[semithick, dashed] (Pj) -- (Pl);
\draw[semithick, dashed] (Pl) -- (Pi);

\foreach \x in {Pi,Pk,Pj,Pl}
{
    \fill (\x) circle (3.2pt);
}

\node[xshift=8pt,yshift=5pt] at (Pi) {$i$};
\node[xshift=10pt,yshift=0pt]  at (Pk) {$k$};
\node[xshift=3pt,yshift=-10pt] at (Pj) {$j$};
\node[xshift=-8pt,yshift=-8pt] at (Pl) {$\ell$};

\node[font=\small] at (1.40, 0.25) {$a_1$};
\node[font=\small] at (0.75, -1.12) {$a_2$};
\node[font=\small] at (-0.45, -1.52) {$a_3$};
\node[font=\small] at (-0.55, 0.02) {$a_4$};
\node[font=\small] at (0.55, 0.15) {$r$};
\node[font=\small] at (-0.55, -0.92) {$r'$};

\end{tikzpicture}
\caption{Two crossing reflections $r = (i,j)$ and $r' = (k,\ell)$ in type~$A_n$, and the reflections $a_1, \dots, a_4$ of \Cref{lem:square} (dashed).}
\label{fig:typeA-crossing-reflections}
\end{figure}
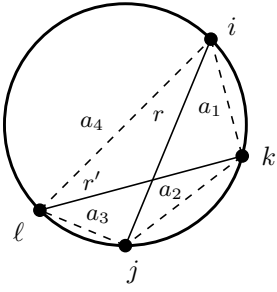

\begin{lemma}
    \label{lem:square}
    Let $r, r', a_1, a_2, a_3, a_4$ be reflections with $a_2, a_4 \leq wr$ and $a_1, a_3 \leq wr'$.
    Suppose that $a_1 a_2, a_2 a_3, a_3 a_4$, and $a_4 a_1$
    are rank-two elements of $[1,w]$ of type $A_2$, and that the reflections below them are $a_1, a_2, r$; $a_2, a_3, r'$; $a_3, a_4, r$; and $a_4, a_1, r'$, respectively.
    Then $r$ and $r'$ are not both $w$-terminal.
\end{lemma}

\begin{proof}
    Suppose that $r$ and $r'$ are both $w$-terminal.
    The three maximal chains of $[1, a_1a_2]$ are labeled $(a_1, a_2)$, $(a_2, r)$, and $(r, a_1)$ (see \Cref{rmk:rank-two}).
    The second one is rising, because $a_2 \leq wr$ and $r$ is $w$-terminal; by $w$-compatibility, the first one is not, i.e., $a_2 \prec a_1$.
    In the same way, the elements $a_2a_3$, $a_3a_4$, and $a_4a_1$ give $a_3 \prec a_2$, $a_4 \prec a_3$, and $a_1 \prec a_4$, a contradiction.
\end{proof}

We now prove \Cref{thm:main-shellability-refined} for the infinite families.

\begin{theorem}
    Let $W$ be a finite Coxeter group of type $A_{n}$, $B_n$, or $D_n$, let \(w\in W\) be a quasi-Coxeter element, and let \(\prec\) be a \(w\)-compatible total order on \(\mathcal R\).
    Then the natural edge labeling of \([1,w]\), with label set ordered by \(\prec\), is an EL-labeling.
    \label{thm:compatible-EL-ABD}
\end{theorem}

\begin{proof}
    We work by induction on the rank $n$.
    By \Cref{rmk:rising.chain,rmk:interval.isomorphism}, it is sufficient to show that every interval of the form $[1, u] \subseteq [1, w]$ has at most one rising maximal chain.
    This trivially holds for rank $n \leq 2$ thanks to the $w$-compatibility condition.

    For $n \geq 3$, consider a reflection $r \in \refl$.
    By \cite[Corollary 6.11]{baumeister2015hurwitzactionfinitecoxeter}, the coatom $wr$ is a quasi-Coxeter element of a parabolic subgroup $W' \subseteq W$.
    Since $W$ is of type $A$, $B$, or $D$, we have $W' = W_1 \times \dots \times W_m$ with each $W_i$ irreducible of type $A$, $B$, or $D$ and of rank $< n$; correspondingly, $wr = w_1 \cdots w_m$ with $w_i$ a quasi-Coxeter element of $W_i$, and $[1,wr] \cong [1,w_1] \times \dots \times [1,w_m]$.
    Note that this product respects the natural edge labeling.
    Moreover, the restriction of $\prec$ to the reflections of $W_i$ is $w_i$-compatible, since every rank-two interval of $[1,w_i]$ is also a rank-two interval of $[1,w]$.
    By induction, every interval in each $[1, w_i]$ has at most one rising maximal chain, and so this is also true for every interval in $[1, wr]$.
    By letting $r$ vary, we obtain that every interval $[1,u] \subseteq [1, w]$ with $u \neq w$ has at most one rising maximal chain.

    It only remains to show that the full interval $[1,w]$ has at most one rising maximal chain.
    Suppose by contradiction that there are two such chains.
    They must pass through different coatoms $wr$ and $wr'$, because we already know that there is only one rising maximal chain below any coatom $wr$.
    By the discussion after \Cref{def:terminal}, $r$ and $r'$ are both $w$-terminal.

    We claim that $r$ and $r'$ are $w$-crossing, i.e., they do not occur together in any reduced reflection factorization of $w$.
    Indeed, suppose that they occur together.
    After interchanging \(r\) and \(r'\) if necessary, and applying Hurwitz moves, we may write $w=u r r'$ with $\ell_\refl(u) = n-2$.
    Set $t = rr'r$.
    Then we have $wr' = ur$ and $wr = ut$; we deduce that $r \leq wr'$ and $t \leq wr$, and so $t \prec r \prec r'$.
    Therefore, the rank-two interval $[1, rr']$ contains the two rising maximal chains labeled $(r, r')$ and $(t, r)$, contradicting $w$-compatibility.

    To conclude the proof, we show, case by case, that there are no two $w$-crossing $w$-terminal reflections $r, r'$.

    \textbf{Case $A_n$.}
    Up to conjugation, $w = (1, 2, \dots, n+1)$.
    Since $r$ and $r'$ are $w$-crossing, we have $r = (i, j)$ and $r' = (k, \ell)$ with $i, k, j, \ell$ all distinct and, possibly after exchanging $k$ and $\ell$, appearing in this cyclic order along $w$.
    Indeed, if $r$ and $r'$ shared an endpoint, or had disjoint and non-interleaved endpoints, then $rr'$ or $r'r$ would be a $3$-cycle whose entries appear in the cyclic order of $w$, or a product of two noncrossing transpositions; in both cases it would be an element of rank $2$ below $w$, and so $r$ and $r'$ would occur together in a reduced reflection factorization of $w$.
    The sides $a_1 = (i,k)$, $a_2 = (k,j)$, $a_3 = (j,\ell)$, and $a_4 = (\ell,i)$ of the quadrilateral in \Cref{fig:typeA-crossing-reflections} satisfy the hypotheses of \Cref{lem:square}, a contradiction.

    \textbf{Case $B_n$.}
    Up to conjugation, $w = [1, 2, \dots, n]$.
    Conjugation by $w$ preserves $[1,w]$ and transforms $\prec$ into another $w$-compatible order; hence, up to rotation (i.e., conjugation by a power of $w$) and possibly exchanging $r$ and $r'$, the reflections $r$ and $r'$ are as in one of the following six cases, with $i<j<k<\ell$ (see \Cref{fig:typeB-crossing-reflections}).
    To determine which reflections lie below $wr$ and $wr'$, note that, for $1 \leq a \leq n$ and $a < b \leq n$,
    \begin{align*}
        w \cdot [a] &= ((1, \dots, a, -(a+1), \dots, -n)), \\
        w \cdot ((a,b)) &= [1, \dots, a, b+1, \dots, n]((a+1, \dots, b)), \\
        w \cdot ((a,-b)) &= ((1, \dots, a, -(b+1), \dots, -n))[a+1, \dots, b]
    \end{align*}
    (where some of these lists may be empty).
    We denote by $w'$ an element of this type.
    A reflection of the form $((x,y))$ is below $w'$ if and only if $x$ and $y$ (or their opposites) are entries of the same factor $((\dots))$, or they are both indices of the unique zero-block, which is of the form $[\dots]$.
    Regarding reflections of the form $[z]$, we have $[z] \le w'$ if and only if $z$ (and hence $-z$) is in the zero-block.

\begin{figure}[t]
\centering
\begin{tikzpicture}[scale=0.8]

\begin{scope}[xshift=0cm]

 \draw[line width=1.1pt] (0,0) circle (2);
 \foreach \i/\lab/\ang in { i/i/90, j/j/45, k/k/0, l/\ell/-45, -i/-i/-90, -j/-j/-135, -k/-k/180, -l/-\ell/135 }{ \coordinate (P\i) at (\ang:2); \fill (P\i) circle (3.2pt); \node at (\ang:2.55) {$\lab$}; }  
 \draw[semithick] (Pi)--(Pk);
 \draw[semithick] (Pj)--(Pl);
 \draw[semithick] (P-i)--(P-k);
 \draw[semithick] (P-j)--(P-l); 
\end{scope}

\begin{scope}[xshift=7cm,yshift=0cm]
 \draw[line width=1.1pt] (0,0) circle (2);
 \foreach \i/\lab/\ang in { i/i/90, j/j/30, -i/-i/-90, -j/-j/-150 }{ \coordinate (P\i) at (\ang:2); \fill (P\i) circle (3.2pt); \node at (\ang:2.55) {$\lab$}; } 

 \draw[semithick] (Pi)--(Pj);
 \draw[semithick] (Pj)--(P-i);
 \draw[semithick] (P-i)--(P-j);
 \draw[semithick] (Pi)--(P-j);
\end{scope}

\begin{scope}[xshift=0cm,yshift=-6.1cm]
 \draw[line width=1.1pt] (0,0) circle (2);
 \foreach \i/\lab/\ang in { i/i/90, j/j/30, k/k/-30, -i/-i/-90, -j/-j/-150, -k/-k/150 }{ \coordinate (P\i) at (\ang:2); \fill (P\i) circle (3.2pt); \node at (\ang:2.55) {$\lab$}; } 

 \draw[semithick] (Pi)--(Pk);
 \draw[semithick] (Pj)--(P-i);
 \draw[semithick] (P-i)--(P-k);
 \draw[semithick] (Pi)--(P-j);
\end{scope}

\begin{scope}[xshift=7cm,yshift=-6.1cm]

 \draw[line width=1.1pt] (0,0) circle (2);
 \foreach \i/\lab/\ang in { i/i/90, j/j/45, k/k/0, l/\ell/-45, -i/-i/-90, -j/-j/-135, -k/-k/180, -l/-\ell/135 }{ \coordinate (P\i) at (\ang:2); \fill (P\i) circle (3.2pt); \node at (\ang:2.55) {$\lab$}; } 

 \draw[semithick] (Pi)--(P-j);
 \draw[semithick] (Pk)--(P-l);
 \draw[semithick] (P-i)--(Pj);
 \draw[semithick] (P-k)--(Pl);
\end{scope}

\begin{scope}[xshift=0cm,yshift=-12.2cm]

 \draw[line width=1.1pt] (0,0) circle (2);
 \foreach \i/\lab/\ang in { i/i/90, j/j/30, k/k/-30, -i/-i/-90, -j/-j/-150, -k/-k/150 }{ \coordinate (P\i) at (\ang:2); \fill (P\i) circle (3.2pt); \node at (\ang:2.55) {$\lab$}; } 

 \draw[semithick] (Pi)--(P-i);
 \draw[semithick] (Pj)--(P-k);
 \draw[semithick] (Pk)--(P-j);
\end{scope}

\begin{scope}[xshift=7cm,yshift=-12.2cm]

 \draw[line width=1.1pt] (0,0) circle (2);
 \foreach \i/\lab/\ang in { i/i/90, j/j/30, -i/-i/-90, -j/-j/-150 }{ \coordinate (P\i) at (\ang:2); \fill (P\i) circle (3.2pt); \node at (\ang:2.55) {$\lab$}; } 

 \draw[semithick] (Pi)--(P-i);
 \draw[semithick] (Pj)--(P-j);
\end{scope}

\end{tikzpicture}
\caption{The six cases of $w$-crossing reflections in type $B_{n}$ (Cases 1-6, from left to right and from top to bottom).}
\label{fig:typeB-crossing-reflections}
\end{figure}
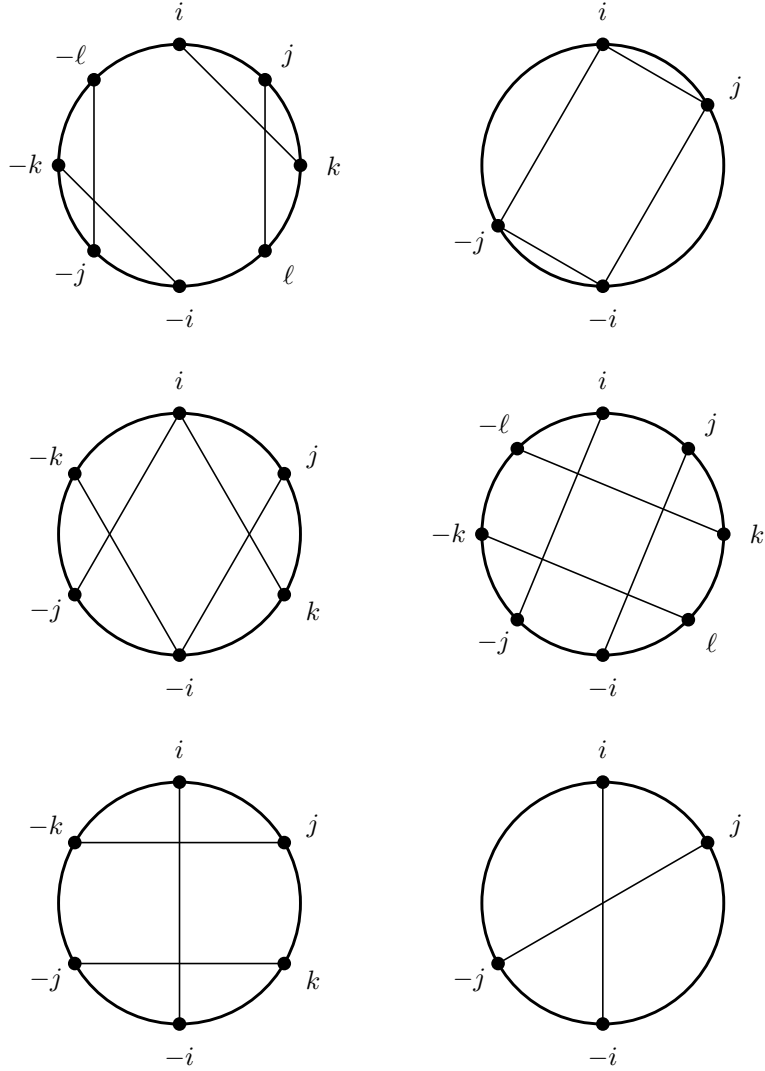
    
    \textbf{Case 1.} Consider $r=((i,k))$ and $r'=((j,\ell))$.
    As in type $A$, the reflections $a_1 = ((i,j))$, $a_2 = ((j,k))$, $a_3 = ((k,\ell))$, and $a_4 = ((i,\ell))$ satisfy the hypotheses of \Cref{lem:square}, a contradiction.

    \textbf{Case 2.} Consider $r=((i,j))$ and $r'=((i,-j))$.
    We have $[i] \leq wr$ and $[j] \leq wr'$; since $r$ and $r'$ are $w$-terminal, it follows that $[i] \prec r$ and $[j] \prec r'$.
    The element $u = [i,j] \in [1,w]$ satisfies $u = [i]((i,j)) = [j]((i,-j))$, so the maximal chains of $[1,u]$ labeled $([i], ((i,j)))$ and $([j], ((i,-j)))$ are both rising, contradicting $w$-compatibility.

    \textbf{Case 3.} Consider $r=((i,k))$ and $r'=((i,-j))$.
    By $w$-compatibility at the indicated rank-two elements, we have
    \begin{alignat*}{2}
        [i] &\prec ((i,-k)) &\qquad& \text{at $[i,k]$, since $[i] \leq wr$} \\
        &\prec ((j,k)) &\qquad& \text{at $((i,-j,-k))$, since $((i,-k)) \leq wr'$} \\
        &\prec ((i,j)) &\qquad& \text{at $((i,j,k))$, since $((j,k)) \leq wr$} \\
        &\prec [i] &\qquad& \text{at $[i,j]$, since $[j] \leq wr'$,}
    \end{alignat*}
    a contradiction.

    \textbf{Case 4.} Consider $r=((i,-j))$ and $r'=((k,-\ell))$.
    By $w$-compatibility at the indicated rank-two elements, we have
    \begin{alignat*}{2}
        [j] &\prec ((i,j)) \prec [i] &\qquad& \text{at $[i,j]$, since $[j] \leq wr$,} \\
        ((i,-k)) &\prec ((j,k)) &\qquad& \text{at $((i,-j,-k))$, since $((i,-k)) \leq wr$,} \\
        ((i,-\ell)) &\prec ((j,\ell)) &\qquad& \text{at $((i,-j,-\ell))$, since $((i,-\ell)) \leq wr$,} \\
        [\ell] &\prec ((k,\ell)) \prec [k] &\qquad& \text{at $[k,\ell]$, since $[\ell] \leq wr'$,} \\
        ((i,k)) &\prec ((i,-\ell)) &\qquad& \text{at $((i,k,-\ell))$, since $((i,k)) \leq wr'$,} \\
        ((j,k)) &\prec ((j,-\ell)) &\qquad& \text{at $((j,k,-\ell))$, since $((j,k)) \leq wr'$.}
    \end{alignat*}
    We now show that $[j] \prec ((j,\ell))$ and $[\ell] \prec ((j,-\ell))$.
    This yields the desired contradiction: the element $u = [j,\ell] \in [1,w]$ satisfies $u = [j]((j,\ell)) = [\ell]((j,-\ell))$, so the maximal chains of $[1,u]$ labeled $([j], ((j,\ell)))$ and $([\ell], ((j,-\ell)))$ would both be rising, contradicting $w$-compatibility.
    We distinguish two subcases, according to the relative order of $((i,j))$ and $((i,-k))$.

    \textbf{Case 4.1.} $((i,j)) \prec ((i,-k))$.
    Then $((i,j)) \prec ((i,-k)) \prec ((j,k)) \prec ((j,-\ell))$, so $w$-compatibility at $((i,j,-\ell))$ gives $((i,j)) \prec ((i,-\ell))$.
    From $[j] \prec ((i,j)) \prec ((i,-k)) \prec ((j,k))$ and $w$-compatibility at $[j,k]$, we get $[k] \prec ((j,k))$.
    Altogether, $[j] \prec ((i,j)) \prec ((i,-\ell)) \prec ((j,\ell))$ and $[\ell] \prec [k] \prec ((j,k)) \prec ((j,-\ell))$, as claimed.

    \textbf{Case 4.2.} $((i,-k)) \prec ((i,j))$.
    Then $((i,-k)) \prec ((i,j)) \prec [i]$, so $w$-compatibility at $[i,k]$ gives $[k] \prec ((i,k))$.
    Applying $w$-compatibility at $[i,\ell]$ to the inequalities $[\ell] \prec [k] \prec ((i,k)) \prec ((i,-\ell))$ yields $[i] \prec ((i,-\ell))$.
    We now have $((k,\ell)) \prec [k] \prec ((i,k)) \prec ((i,-\ell)) \prec ((j,\ell))$, and hence $((k,\ell)) \prec ((j,k))$ by $w$-compatibility at $((j,k,\ell))$.
    Altogether, $[j] \prec [i] \prec ((i,-\ell)) \prec ((j,\ell))$ and $[\ell] \prec ((k,\ell)) \prec ((j,k)) \prec ((j,-\ell))$, as claimed.

    \textbf{Case 5.} Consider $r=[i]$ and $r'=((j,-k))$.
    By $w$-compatibility at the indicated rank-two elements, we have
    \begin{alignat*}{2}
        [j] &\prec ((i,j)) &\qquad& \text{at $[i,j]$, since $((i,-j)) \leq wr$} \\
        &\prec ((i,-k)) &\qquad& \text{at $((i,j,-k))$, since $((i,j)) \leq wr'$} \\
        &\prec [k] &\qquad& \text{at $[i,k]$, since $((i,-k)) \leq wr$} \\
        &\prec [j] &\qquad& \text{at $[j,k]$, since $[k] \leq wr'$,}
    \end{alignat*}
    a contradiction.

    \textbf{Case 6.} Consider $r=[i]$ and $r'=[j]$.
    We have $((i,-j)) \leq wr$ and $((i,j)) \leq wr'$; since $r$ and $r'$ are $w$-terminal, it follows that $((i,-j)) \prec r$ and $((i,j)) \prec r'$.
    The element $u = [i,j] \in [1,w]$ satisfies $u = ((i,-j))[i] = ((i,j))[j]$, so the maximal chains of $[1,u]$ labeled $(((i,-j)), [i])$ and $(((i,j)), [j])$ are both rising, contradicting $w$-compatibility.

    \textbf{Case $D_n$.}
    Up to conjugation, $w = \wpq$, and we use the annular description of $[1,w]$ from \Cref{sec:annular}:
    a reflection is \emph{outer} or \emph{inner} if both its indices lie on the outer or on the inner circle, respectively, and a \emph{connectivity reflection} otherwise.
    An outer and an inner reflection are never $w$-crossing, since their product is a rank-two element below $w$.
    If $r$ and $r'$ are outer (the case of two inner reflections is analogous), then, up to rotation of the outer circle and possibly exchanging $r$ and $r'$, they are as in Cases 1-4 of type $B$, with $1 \leq i<j<k<\ell \leq p$.
    The coatoms $wr$ and $wr'$ are then given by the formulas of type $B$, with $n$ replaced by $p$ and the factor $[p+1, \dots, n]$ appended, and the same membership rule applies.
    We also fix $h \in \{p+1, \dots, n\}$.

    \textbf{Case 1.} Consider $r=((i,k))$ and $r'=((j,\ell))$.
    As in Case 1 of type $B$, the reflections $a_1 = ((i,j))$, $a_2 = ((j,k))$, $a_3 = ((k,\ell))$, and $a_4 = ((i,\ell))$ satisfy the hypotheses of \Cref{lem:square}, a contradiction.

    \textbf{Case 2.} Consider $r=((i,j))$ and $r'=((i,-j))$.
    Then the reflections $a_1 = ((j,h))$, $a_2 = ((i,h))$, $a_3 = ((j,-h))$, and $a_4 = ((i,-h))$ satisfy the hypotheses of \Cref{lem:square}, a contradiction.

    \textbf{Case 3.} Consider $r=((i,k))$ and $r'=((i,-j))$.
    The element $[p+1, \dots, n]^q$ maps each $x \in \{p+1, \dots, n\}$ to $-x$ and fixes $1, \dots, p$, so conjugation by it fixes $w$, $r$, and $r'$, and exchanges $((i,h))$ and $((i,-h))$.
    Since this conjugation also transforms $\prec$ into another $w$-compatible order, we may assume that $((i,-h)) \prec ((i,h))$.
    By $w$-compatibility at the indicated rank-two elements, we have
    \begin{alignat*}{2}
        ((j,k)) &\prec ((i,j)) &\qquad& \text{at $((i,j,k))$, since $((j,k)) \leq wr$,} \\
        ((i,h)) &\prec ((k,h)) &\qquad& \text{at $((i,k,h))$, since $((i,h)) \leq wr$,} \\
        ((i,-k)) &\prec ((j,k)) &\qquad& \text{at $((i,-j,-k))$, since $((i,-k)) \leq wr'$,} \\
        ((j,h)) &\prec ((i,-h)) &\qquad& \text{at $((i,-h,-j))$, since $((j,h)) \leq wr'$.}
    \end{alignat*}
    Then $((j,h)) \prec ((i,-h)) \prec ((i,h))$, so $w$-compatibility at $((i,j,h))$ gives $((i,j)) \prec ((i,h))$.
    From $((i,-h)) \prec ((i,h)) \prec ((k,h))$ and $w$-compatibility at $((i,-h,-k))$, we get $((i,-h)) \prec ((i,-k))$.
    Applying $w$-compatibility at $((j,k,h))$ to the inequalities $((j,h)) \prec ((i,-h)) \prec ((i,-k)) \prec ((j,k))$ yields $((k,h)) \prec ((j,k))$.
    Altogether, $((j,k)) \prec ((i,j)) \prec ((i,h)) \prec ((k,h)) \prec ((j,k))$, a contradiction.

    \textbf{Case 4.} Consider $r=((i,-j))$ and $r'=((k,-\ell))$.
    As in Case 3, we may assume that $((i,-h)) \prec ((i,h))$.
    By $w$-compatibility at the indicated rank-two elements, we have
    \begin{alignat*}{2}
        ((i,-k)) &\prec ((j,k)) &\qquad& \text{at $((i,-j,-k))$, since $((i,-k)) \leq wr$,} \\
        ((i,-\ell)) &\prec ((j,\ell)) &\qquad& \text{at $((i,-j,-\ell))$, since $((i,-\ell)) \leq wr$,} \\
        ((j,h)) &\prec ((i,-h)) &\qquad& \text{at $((i,-h,-j))$, since $((j,h)) \leq wr$,} \\
        ((i,k)) &\prec ((i,-\ell)) &\qquad& \text{at $((i,k,-\ell))$, since $((i,k)) \leq wr'$,} \\
        ((j,k)) &\prec ((j,-\ell)) &\qquad& \text{at $((j,k,-\ell))$, since $((j,k)) \leq wr'$,} \\
        ((\ell,-h)) &\prec ((k,h)) &\qquad& \text{at $((k,h,-\ell))$, since $((\ell,-h)) \leq wr'$.}
    \end{alignat*}
    We distinguish two subcases, according to the relative order of $((i,h))$ and $((j,\ell))$.

    \textbf{Case 4.1.} $((j,\ell)) \prec ((i,h))$.
    Then $((j,h)) \prec ((i,-h)) \prec ((i,h))$, so $w$-compatibility at $((i,j,h))$ gives $((j,h)) \prec ((i,j)) \prec ((i,h))$.
    From $((i,-\ell)) \prec ((j,\ell)) \prec ((i,h))$ and $w$-compatibility at $((i,h,-\ell))$, we get $((i,-\ell)) \prec ((\ell,-h))$.
    Next, $((i,k)) \prec ((i,-\ell)) \prec ((\ell,-h)) \prec ((k,h))$, so $w$-compatibility at $((i,k,h))$ gives $((i,h)) \prec ((k,h))$.
    Applying $w$-compatibility at $((i,-h,-k))$ to the inequalities $((i,-h)) \prec ((i,h)) \prec ((k,h))$ yields $((i,-h)) \prec ((i,-k))$.
    We now have $((j,h)) \prec ((i,-h)) \prec ((i,-k)) \prec ((j,k))$, and hence $((k,h)) \prec ((j,k))$ by $w$-compatibility at $((j,k,h))$.
    Then $((i,j)) \prec ((i,h)) \prec ((k,h)) \prec ((j,k)) \prec ((j,-\ell))$, so $w$-compatibility at $((i,j,-\ell))$ gives $((i,j)) \prec ((i,-\ell))$.
    Finally, from $((\ell,-h)) \prec ((k,h)) \prec ((j,k)) \prec ((j,-\ell))$ and $w$-compatibility at $((j,h,-\ell))$, we get $((\ell,-h)) \prec ((j,h))$.
    Altogether, $((i,-\ell)) \prec ((\ell,-h)) \prec ((j,h)) \prec ((i,j)) \prec ((i,-\ell))$, a contradiction.

    \textbf{Case 4.2.} $((i,h)) \prec ((j,\ell))$.
    Then $((i,-h)) \prec ((i,h)) \prec ((j,\ell))$, so both $((i,h))$ and $((i,-h))$ precede $((j,\ell))$.
    From now on, we only use this fact and the six inequalities above.
    The conjugation of Case 3 preserves this fact and exchanges $((\ell,h))$ and $((\ell,-h))$, and the six inequalities still hold afterwards, since they only use that $r$ and $r'$ are $w$-terminal; hence we may also assume that $((\ell,h)) \prec ((\ell,-h))$.
    Then $((j,h)) \prec ((i,-h)) \prec ((j,\ell))$, so $w$-compatibility at $((j,\ell,h))$ gives $((j,h)) \prec ((\ell,h))$.
    Applying $w$-compatibility at $((k,\ell,h))$ to the inequalities $((\ell,h)) \prec ((\ell,-h)) \prec ((k,h))$ yields $((\ell,h)) \prec ((k,\ell)) \prec ((k,h))$.
    From $((j,h)) \prec ((\ell,h)) \prec ((\ell,-h))$ and $w$-compatibility at $((j,h,-\ell))$, we get $((j,-\ell)) \prec ((\ell,-h))$.
    Next, $((j,k)) \prec ((j,-\ell)) \prec ((\ell,-h)) \prec ((k,h))$, so $w$-compatibility at $((j,k,h))$ gives $((j,k)) \prec ((j,h))$.
    We now have $((i,-k)) \prec ((j,k)) \prec ((j,h)) \prec ((i,-h))$, and hence $((k,h)) \prec ((i,-h))$ by $w$-compatibility at $((i,-h,-k))$.
    From $((j,k)) \prec ((j,h)) \prec ((\ell,h)) \prec ((k,\ell))$ and $w$-compatibility at $((j,k,\ell))$, we get $((j,\ell)) \prec ((k,\ell))$.
    Altogether, $((j,\ell)) \prec ((k,\ell)) \prec ((k,h)) \prec ((i,-h)) \prec ((j,\ell))$, a contradiction.

    Finally, suppose that $r$ or $r'$ is a connectivity reflection.
    As in type $A$, $r$ and $r'$ are represented in the annulus by two crossing chords, which are the diagonals of a quadrilateral (see \Cref{fig:typeD-crossing-reflections}); more precisely, we can write $r = ((x,y))$ and $r' = ((z,v))$ so that $((x,z,y,v)) \in [1,w]$.
    The sides $a_1 = ((x,z))$, $a_2 = ((z,y))$, $a_3 = ((y,v))$, and $a_4 = ((v,x))$ of this quadrilateral satisfy the hypotheses of \Cref{lem:square}, a contradiction.
\end{proof}

\begin{figure}[t]
\centering
\begin{tikzpicture}[scale=0.7]

\begin{scope}[xshift=0cm, scale=0.7]

\def\R{4}
\def\r{1.35}

\draw[line width=1.1pt] (0,0) circle (\R);
\draw[line width=1.1pt] (0,0) circle (\r);

\coordinate (P1)  at (90:\R);
\coordinate (P2)  at (30:\R);
\coordinate (P3)  at (-30:\R);
\coordinate (Pm1) at (-90:\R);
\coordinate (Pm2) at (-150:\R);
\coordinate (Pm3) at (150:\R);

\coordinate (P4)  at (90:\r);
\coordinate (Pm5) at (0:\r);
\coordinate (Pm4) at (-90:\r);
\coordinate (P5)  at (180:\r);

\foreach \x in {P1,P2,P3,Pm1,Pm2,Pm3,P4,Pm5,Pm4,P5}
{
    \fill (\x) circle (4.5pt);
}

\node[yshift=12pt] at (P1) {$1$};
\node[xshift=12pt] at (P2) {$2$};
\node[xshift=12pt] at (P3) {$3$};

\node[yshift=-14pt] at (Pm1) {$-1$};
\node[xshift=-14pt] at (Pm2) {$-2$};
\node[xshift=-14pt] at (Pm3) {$-3$};

\node[yshift=-10pt] at (P4) {$4$};
\node[xshift=16pt] at (Pm5) {$-5$};
\node[yshift=10pt] at (Pm4) {$-4$};
\node[xshift=-14pt] at (P5) {$5$};

 \draw[semithick] (P1)--(P4);
 \draw[semithick] (Pm1)--(Pm4);
 \draw[semithick] (P2)--(Pm3);
 \draw[semithick] (Pm2)--(P3);

\end{scope}

\begin{scope}[xshift=7.5cm, scale=0.7]

\def\R{4}
\def\r{1.35}

\draw[line width=1.1pt] (0,0) circle (\R);
\draw[line width=1.1pt] (0,0) circle (\r);

\coordinate (P1)  at (90:\R);
\coordinate (P2)  at (30:\R);
\coordinate (P3)  at (-30:\R);
\coordinate (Pm1) at (-90:\R);
\coordinate (Pm2) at (-150:\R);
\coordinate (Pm3) at (150:\R);

\coordinate (P4)  at (90:\r);
\coordinate (Pm5) at (0:\r);
\coordinate (Pm4) at (-90:\r);
\coordinate (P5)  at (180:\r);

\foreach \x in {P1,P2,P3,Pm1,Pm2,Pm3,P4,Pm5,Pm4,P5}
{
    \fill (\x) circle (4.5pt);
}

\node[yshift=12pt] at (P1) {$1$};
\node[xshift=12pt] at (P2) {$2$};
\node[xshift=12pt] at (P3) {$3$};

\node[yshift=-14pt] at (Pm1) {$-1$};
\node[xshift=-14pt] at (Pm2) {$-2$};
\node[xshift=-14pt] at (Pm3) {$-3$};

\node[yshift=10pt] at (P4) {$4$};
\node[xshift=16pt] at (Pm5) {$-5$};
\node[yshift=-12pt] at (Pm4) {$-4$};
\node[xshift=-14pt] at (P5) {$5$};

\draw[black, semithick] (P1) .. controls (2.1,1.8) and (2.1,-0.2) .. (Pm5);

\draw[semithick] (P2)--(P4);

\draw[black, semithick] (Pm1) .. controls (-1.5,-1.8) and (-1.7,-0.6) .. (P5);

\draw[semithick] (Pm2)--(Pm4);

\end{scope}

\end{tikzpicture}
\caption{Two pairs of $w_{3,2}$-crossing reflections in type $D_{5}$: $((1,4))$ and $((2,-3))$ (left), and $((1,-5))$ and $((2,4))$ (right).}
\label{fig:typeD-crossing-reflections}
\end{figure}
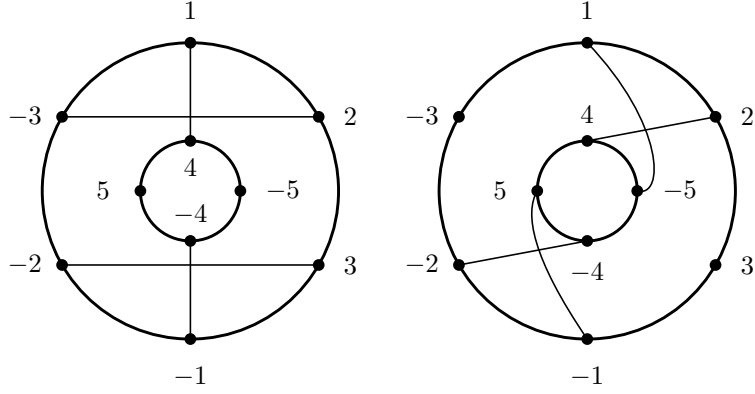

\Cref{thm:compatible-EL-ABD} shows that \(w\)-compatibility is a
necessary and sufficient condition for a total order on the reflections to
make the natural edge labeling an EL-labeling in the infinite families $A$, $B$, and $D$.
In types \(A\) and \(B\), every quasi-Coxeter element is a Coxeter element,
and the existence of $w$-compatible orders follows from
\Cref{thm:EL-labeling.for.Coxeter.elements}. It remains to treat the proper
quasi-Coxeter elements of type \(D\).

\begin{proposition}
    For every \(p,q\geq 1\), there exists a \(w_{p,q}\)-compatible total order on~\(\refl\).
    \label{prop:D-compatible-order}
\end{proposition}

\begin{proof}
    We claim that a feasible choice of orders in rank-two intervals is given in \Cref{tab:local-rules}.
    Indeed, the following total order agrees with all local orders:
    \begin{align*}
    &((1, 2)) \prec ((1, 3)) \prec 
    \dots \prec ((1, n))\prec 
     ((1, -2)) \prec \dots \prec ((1, -p)) \\
    &\prec ((2, 3)) \prec \dots \prec ((2, n)) 
    \prec((2, -3)) \prec \dots \prec ((2, -p)) \\
    &\prec \dots \prec ((p, p+1)) \prec \dots \prec ((p, n)) \prec((1,-(p+1))) \prec \dots \prec ((p, -(p+1))) \\
    &\prec ((p+1, -(p+2)))
    \prec ((1, -(p+2))) \prec \dots \prec ((p, -(p+2))) \\ 
    &\prec ((p+1, p+2))\prec ((p+1, -(p+3))) \prec ((p+2, -(p+3))) \\
    &\prec((1, -(p+3))) \prec \dots \prec ((p, -(p+3))) \\  
    &\prec ((p+1, p+3)) \prec ((p+2, p+3))\prec \dots \prec ((n - 1, n)). \qedhere
    \end{align*}
\end{proof}

\begin{table}[t]
    \centering
    \begin{tabular}{c|c|c}
        $u$ & positions of $i < j < k$ & local order  \\
        \hline
        $((i, j, k))$ & $OOO$, $OOI$, $OII$, $III$ & $((i, j)) \prec ((i, k)) \prec ((j, k))$ \\
        $((i, j, -k))$ & $OOO$, $OOI$, $III$ & $((i, j)) \prec ((i,-k)) \prec ((j,-k))$ \\
        $((i, k, -j))$ & $OOI$, $OII$ & $((i, k)) \prec ((i,-j)) \prec ((j,-k))$ \\
        $((i, -k, j))$ & $OII$ & $((i, j)) \prec ((j,-k)) \prec ((i,-k))$ \\
        $((i, -j, -k))$ & $OOO$, $OII$, $III$ & $((i,-j)) \prec ((i,-k)) \prec ((j,k))$ \\
        $((i, -k, -j))$ & $OOI$ & $((i, -j)) \prec ((j,k)) \prec ((i,-k))$ 
    \end{tabular}
    \caption{
        A choice of orders in irreducible rank-two intervals $[1, u]$ extending to a global order on $\refl$ in type $D_n$.
        In each row, we have $1 \leq i < j < k \leq n$; the second column records the possible positions of $i, j, k$ in the outer and inner circles $O=\{1,\ldots,p\}$ and $I=\{p+1,\ldots,n\}$.
    }
    \label{tab:local-rules}
\end{table}

This completes the proof of \Cref{thm:main-shellability} for the three infinite families.

\begin{remark}
    Unlike the case of Coxeter elements \cite{AthanasiadisBradyWatt2005} (for a suitable choice of positive roots), for proper quasi-Coxeter elements of type $D$, no $w_{p,q}$-compatible order is a reflection ordering (in the sense of Dyer \cite{dyer1993hecke}) with respect to the standard set of positive roots
    \[ \Phi_{D}^{+}= 
    \{e_{i} \pm e_{j} \mid 1 \le i<j \le n\}.
    \]
    Indeed, the reflection ordering condition, together with $w_{p,q}$-compatibility, forces the total order in every irreducible rank-two interval $[1, u]$.
    For example, for $w_{2,2} = [1,2][3,4]$ in $D_4$, we have:
    \begin{align*}
       u = ((1,-4,3)): &\qquad ((3,-4))\prec((1,-4))\prec((1,3)) \\
       u = ((1,-4,-2)): &\qquad ((1,-4))\prec((1,-2))\prec((2,4)) \\
       u = ((1,3,-2)): &\qquad ((1,3))\prec((1,-2))\prec((2,-3)) \\
       u = ((2,4,-3)): &\qquad ((2,4))\prec((2,-3))\prec((3,-4)),
    \end{align*}
    implying the contradiction
    \[ ((3,-4))\prec((1,-4))\prec((1,-2))\prec((2,-3))\prec((3,-4)). \]
    The same reasoning applies to the interval below $[1,2][p+1,p+2]$ in any interval $\interpq$, with $p,q \geq 2$.
\end{remark}

\begin{remark}\label{rem:thomas.proof}
  The following result, due to Thomas (private communication), reduces the first part of \Cref{thm:main-shellability-refined} to a finite check.
  Let $w$ be a quasi-Coxeter element of a finite Coxeter group $W$ of rank at least $2$. We denote by $m$ the maximum rank among all minimal common upper bounds of pairs of distinct atoms in the interval $\inter$.
  Suppose that, for every subinterval $[1,w'] \subseteq \inter$ with $\operatorname{rank}(w') \le m$, all $w'$-compatible orders yield EL-labelings of $[1,w']$.
  Then every $w$-compatible order yields an EL-labeling of $\inter$.
\end{remark}

\begin{proof}
    Let $\prec$ be a $w$-compatible order. We show, by induction on $\operatorname{rank}(x)$, that the natural edge labeling of $[1,x]$ is an EL-labeling for every $x \in \inter$; this holds for $\operatorname{rank}(x) \leq m$ by assumption, since $\prec$ is also $x$-compatible. By \Cref{rmk:interval.isomorphism}, every proper subinterval $[a,b]$ of $[1,x]$ is label-preservingly isomorphic to $[1,a^{-1}b]$, where $a^{-1}b \le b \le x$, so the induction hypothesis applies. It remains to show that $[1,x]$ has a unique rising maximal chain, and that it is lexicographically smallest.
    We know by \Cref{rmk:rising.chain} that the lexicographically smallest chain $c$ of $[1,x]$ is rising. Let $s$ be its first reflection, which is the least reflection below $x$.
    So it suffices to show that $[1,x]$ has no other rising maximal chain.
    We prove this by contradiction: suppose that there exists another rising chain $c'$, which necessarily starts from some other reflection, say $t$.  

Choose any minimal common upper bound $u$ of $s$ and $t$ in $[1,x]$. It is also minimal in $\inter$, so its rank is at most $m$. Consider the only rising chain $d$ of the subinterval $[t,u]$.
By the assumption that the result holds for intervals of rank up to $m$, the chain obtained by adding $t$ before $d$ is not rising in the interval $[1,u]$, since the only rising chain in that interval begins with the reflection $s$. Thus the first reflection $r'$ of $d$ precedes $t$.

Now consider the only rising chain $d'$ in $[t,x]$. By induction, it is unique and lexicographically smallest.
Extending $d$ to a maximal chain of $[t,x]$ and using lexicographic minimality shows that the first reflection of $d'$ is at most $r'$ with respect to $\prec$. In particular, the first reflection of $d'$ precedes $t$. The only way to possibly obtain a rising chain of the entire interval $[1,x]$ starting with $t$ is to have $t$ followed by $d'$, but we have just shown that this chain is not rising.
So the only rising chain in $[1,x]$ is $c$, as desired.
\end{proof}

In the infinite families $A_{n}$, $B_{n}$, and $D_{n}$, the finite check required by \Cref{rem:thomas.proof} is equivalent to the explicit one carried out in the proof of \Cref{thm:compatible-EL-ABD}.
In the exceptional types, the full statement already amounts to a finite check, as there are only finitely many intervals to consider;
we treat the exceptional types in \Cref{sec:exceptional-cases}.

\section{Zeta polynomials in type \texorpdfstring{$D$}{D}}
\label{sec:zeta.polynomial}

In this section, our objective is to determine the zeta polynomials of the intervals $\interpq_{D}$, where the element $w_{p,q}$ is defined as in Section \ref{sec:preliminaries}. Throughout this section, let $p$ and $q$ be two positive integers and $n=p+q$. We define the connectivity of an element of $\annularncB(p,q)$.

\begin{definition}\label{def:connectivity}
 For $w \in \NC^{B}(p,q)$, the connectivity of $w$ is the number
\[
c(w) :=
\frac{1}{2}
\left(
\begin{array}{l}
\# \text{ of nonzero blocks } A \text{ of } w  
\text{ such that} \\ 
A \text{ intersects both sets }
\{\pm 1,\ldots,\pm p\} \\
\text{ and }
\{\pm (p+1),\ldots,\pm (p+q)\}
\end{array}
\right).
\]   
\end{definition}

Recall from \Cref{sec:annular} that the interval $\interpq_{D}$ is isomorphic (as a poset) to the annular noncrossing partition poset $\annularncD(p,q)$ \cite[Corollary 7.2]{NICA.annular.noncrossing}.
We compute the zeta polynomial of $\annularncD(p,q)$ starting from the known formula for the zeta polynomial of $\annularncB(p,q)$.

\begin{theorem}[{\cite[Theorem 5.3]{GouldenNicaEnumerativeProperties}}]\label{thm:zeta.polynomial.B}
    Let $p,q$ be positive integers.
    The zeta polynomial of $\NC^{B}(p,q)$ is given by the formula
    \begin{equation}\label{eq:zeta.annular.B}
    Z(\NC^{B}(p,q),m)
    =
    \binom{mp}{p}
    \binom{mq}{q}
    +
    \sum_{c=1}^{p}
    2c
    \binom{mp}{p-c}
    \binom{mq}{q+c}.
    \end{equation}
\end{theorem}

The formula in Theorem \ref{thm:zeta.polynomial.B} is not closed, as it contains a sum over the connectivity parameter $c$.
It is explained in \cite{GouldenNicaEnumerativeProperties} that the first summand of \eqref{eq:zeta.annular.B} counts the number of multichains
$$x_{1} \le \dots \le x_{m-1}$$
of length $m-2$ such that all elements $x_{i}$ of the multichain have connectivity $c=0$. 
We refer to elements of this type as zero-connectivity elements.
Let $w_{p} = [1,\dots,p]$ and $w_{q} = [p+1,\dots,n]$, so that $\wpq = w_{p} w_{q}$.
Since $w_p$ and $w_q$ do not lie in $D_n$, intervals below them are always taken in $B_n$, and we omit the subscript.
We observe that a zero-connectivity element $w$ in the interval $\interpq_{B}$ is the product $w=u \cdot v$ of an element $u \in [1,\wpp]$ and an element $v \in [1,w_{q}]$.
It follows from \cite[Proposition 7]{REINER1997195} that the zeta polynomial of the type $B$ interval $[1,\wpp]$ is
\[
    Z([1,\wpp],m)
    =
    \binom{mp}{p}. \]

The second summand of \eqref{eq:zeta.annular.B} counts the number of multichains such that there is an element $x_{i}$ in the multichain with
$$c(x_{i})>0.$$
In this case, we call $x_{i}$ a positive-connectivity element.
If we set $m=2$, the summand corresponding to $c=k$ counts the number of elements with connectivity $k$.
We first derive an equivalent closed formula for the zeta polynomial $Z(\annularncB(p,q),m)$.

\begin{corollary}\label{cor:zeta.poly.B.closed}
    Let $p,q$ be positive integers.
    The zeta polynomial of $\NC^{B}(p,q)$ is given by the formula
    \[
    Z(\NC^{B}(p,q),m)
    =
    \left(1+ \frac{2(m-1)pq}{m(p+q)}
    \right) 
    \binom{mp}{p}
    \binom{mq}{q}.
    \]
\end{corollary}

\begin{proof}
    By \Cref{thm:zeta.polynomial.B}, it is sufficient to prove that
    \[
    \sum_{c=1}^{p}
    c
    \binom{mp}{p-c}
    \binom{mq}{q+c}
    =
    \frac{(m-1)pq}{m(p+q)}
    \binom{mp}{p}
    \binom{mq}{q}.
    \]
    We show by induction that, for $1 \leq k \leq p+1$, the partial sum on the left-hand side admits the following closed formula:
    \[
    \sum_{c=k}^{p}
    c
    \binom{mp}{p-c}
    \binom{mq}{q+c}
    =
    \frac{(mq - q - k + 1)(p - k + 1)}{m(p+q)}
    \binom{mp}{p-k+1}
    \binom{mq}{q+k-1}.
    \]
    Our induction starts from the base case $k=p+1$, which is trivial since both sides are zero; in the induction step, we assume that the formula above holds for $k+1$ and we prove it for $k$:
    \begin{align*}
        &\sum_{c=k}^{p}
        c
        \binom{mp}{p-c}
        \binom{mq}{q+c} =
        k \binom{mp}{p-k}\!
        \binom{mq}{q+k} +
        \sum_{c=k+1}^{p}
        c
        \binom{mp}{p-c}
        \binom{mq}{q+c} \\
        &\quad =  \left( k + \frac{(mq - q - k)(p - k)}{m(p+q)} \right) \binom{mp}{p-k}
        \binom{mq}{q+k} \\
        &\quad =  \frac{(mp - p + k)(q + k)}{m(p+q)}  \binom{mp}{p-k}
        \binom{mq}{q+k} \\
        &\quad = \frac{(mq - q - k + 1)(p - k + 1)}{m(p+q)}
        \binom{mp}{p-k+1}
        \binom{mq}{q+k-1}.
    \end{align*}
    The desired result follows by setting $k=1$.
\end{proof}

We recall that $\annularncDpq$ is a subposet of $\annularncBpq$.
In order to prove the formula in Theorem \ref{thm:zeta.polynomial.D} for the zeta polynomial of $\annularncDpq$,  we investigate which multichains of $\annularncBpq$ are also multichains of $\annularncDpq$, starting with multichains that contain only zero-connectivity elements.

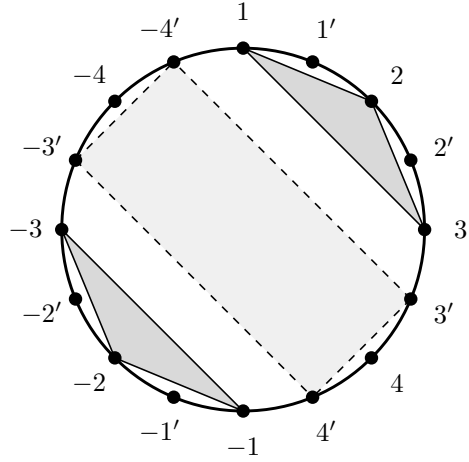
\begin{figure}[t]
\centering

\begin{tikzpicture}[scale=1.2]

\draw[line width=1.1pt] (0,0) circle (2);

\foreach \i/\ang in {
1/90,
1'/67.5,
2/45,
2'/22.5,
3/0,
3'/-22.5,
4/-45,
4'/-67.5,
-1/-90,
-1'/-112.5,
-2/-135,
-2'/-157.5,
-3/180,
-3'/157.5,
-4/135,
-4'/112.5
}{
    \coordinate (P\i) at (\ang:2);
    \node at (\ang:2.4) {$\i$};
}

\fill[gray!30] (P1) -- (P2) -- (P3) -- cycle;
\fill[gray!30] (P-1) -- (P-2) -- (P-3) -- cycle;
\fill[gray!12] (P3') -- (P4') -- (P-3') -- (P-4') -- cycle;

\draw[semithick] (P1)--(P2);
\draw[semithick] (P2)--(P3);
\draw[semithick] (P3)--(P1);

\draw[semithick] (P-1)--(P-2);
\draw[semithick] (P-2)--(P-3);
\draw[semithick] (P-3)--(P-1);

\draw[semithick, dashed] (P3')--(P4');
\draw[semithick, dashed] (P-3')--(P4');
\draw[semithick, dashed] (P-3')--(P-4');
\draw[semithick, dashed] (P3')--(P-4');

\foreach \i in {1,1',2,2',3,3',4,4',-1,-1',-2,-2',-3,-3',-4,-4'}
    \fill (P\i) circle (2.1pt);

\end{tikzpicture}

\caption{The Kreweras complement of $((1,2,3))$ is $[3,4]$ (dashed).}
\label{fig:Kreweras.complement}

\end{figure}

We investigate the structure of the type $B$ interval
$[1, \wpp]$; everything we prove also holds for $[1,w_{q}]$ after relabelling.
First we recall the definition of the Kreweras complement.

\begin{definition}
    Let $W$ be a finite Coxeter group. For all pairs of elements $u,v \in W$, define the map
    $\kre: W \to W$ by setting 
    $\kre(w) \coloneqq u w^{-1} v.$
\end{definition}

If $u \ler v$, then $\kre$ is an anti-automorphism of the interval $[u, v]$, as proved in \cite[Lemma 2.5.4]{armstrong2007generalizednoncrossingpartitionscombinatorics}.
There is an intuitive way to visualize this anti-automorphism for the map $K_{1}^{\wpp}:[1,\wpp] \to [1,\wpp]$, described in \cite[Section 4.2]{armstrong2007generalizednoncrossingpartitionscombinatorics}.
We fix an element $w \in [1,\wpp]$ and consider its classical noncrossing partition structure. Then we add on the circle the points $1',\dots, p',-1',\dots,-p'$, where the point $i'$ is obtained with a clockwise rotation of $\frac{\pi}{2p}$ radians of the point $i$, for every $i \in \{1,\dots,p,-1,\dots,-p\}$. Then $K_{1}^{\wpp}(w)$ corresponds to the coarsest partition of the set $\{1',\dots, p',-1',\dots,-p'\}$ such that no block of $w$ crosses a block of $K_{1}^{\wpp}(w)$.
For example, if $p=4$ and $w=((1,2,3))$, we see in Figure \ref{fig:Kreweras.complement} that $K_{1}^{\wpp}(w)=[3,4]$.
Indeed,
$$w_{4} = [1,2,3,4] = ((1,2,3))[3,4].$$

\begin{remark}\label{rem:kreweras.complement}
    Recall from \Cref{sec:annular} that the elements
    $((i_1,\dots,i_r))$ are even and $[i_1,\dots,i_r]$ are odd.
    Since every element of $[1,\wpp]$ has at most one cycle
    of the latter form, corresponding to its zero-block,
    it is odd if and only if it has a zero-block.
    The element $\wpp$ is odd, so $w$ and
    $K_1^{\wpp}(w)=w^{-1}\wpp$ have opposite parity.
    Hence $K_1^{\wpp}$ is an anti-automorphism which swaps
    the subposets $[1,\wpp]_{\NZ}$ and
    $[1,\wpp]_{\ZZ}$, consisting respectively
    of the elements without and with a zero-block.
    The same argument applies to the map $K_1^b$ on $[1,b]$
    whenever $b \in [1,\wpp]$ has a zero-block, since $b$
    is odd; it also applies to $K_a^{\wpp}$ on $[a,\wpp]$
    whenever $a \in [1,\wpp]$ has no zero-blocks, since $a$ is even.
\end{remark}

Let $m \ge 2$ be an integer and $0 \leq i \leq m-1$.
We denote by $\MC_{p}(m,i)$ the set of multichains 
$$x_{1} \ler \dots \ler x_{m-1}$$
of length $m-2$ in 
$[1,\wpp]$, such that $x_j$ has no zero-blocks for $j \leq i$ and has a zero-block for $j > i$.

\begin{lemma}\label{lem:number.multichains.B}
    The sets $\MC_{p}(m,i)$, for $0 \le i \le m-1$, have the same number of elements:
    \begin{equation}\label{eq:formula.mp}
      |\MC_{p}(m,i)| = \frac{1}{m} \binom{mp}{p}.  
    \end{equation}    
\end{lemma}

\begin{proof}
      We observe that $|\MC_{p}(m,0)|=|\MC_{p}(m,m-1)|$, since we know by Remark \ref{rem:kreweras.complement} that
      $[1,\wpp]_{\NZ}$ and $[1,\wpp]_{\ZZ}$ are anti-isomorphic posets. This means that $K^{\wpp}_{1}$ maps multichains of the first poset bijectively to multichains of the second, reversing their order.
      We extend this result to prove that for every $0 < i < m-1$, we get 
      $|\MC_{p}(m,i)|=|\MC_{p}(m,m-1)|$. In particular, we prove that for every $x \in [1,\wpp]_{\NZ}$, the number of multichains having $x$ as the $i$-th element is the same in 
      $\MC_{p}(m,i)$ and $\MC_{p}(m,m-1)$.
      Let $\mathcal{C} \colon x_{1} \ler \dots \ler x_{i}$ be a multichain of length $i-1$ in $[1,\wpp]_{\NZ}$. 
      Then we consider the interval $[x_{i},\wpp]$.
      Using Remark \ref{rem:kreweras.complement} again, there are as many multichains of length $m-2-i$ in
      $[x_{i},\wpp]_{\NZ}$ as in $[x_{i},\wpp]_{\ZZ}$.
      Hence there are the same number of ways to complete the multichain $\mathcal{C}$ to a multichain in $\MC_{p}(m,i)$ and $\MC_{p}(m,m-1)$.
      Since the $m$ sets $\MC_{p}(m,i)$ partition all multichains of $[1,w_p]$, the equation \eqref{eq:formula.mp} follows from the known formula for the zeta polynomial of type $B$ intervals \cite[Proposition 7]{REINER1997195}. 
\end{proof}

For every natural number $m \geq 2$, we denote by $Z^{0}(\NC^{B}(p,q),m)$ the number of multichains of $\annularncBpq$ of length $m-2$ such that all elements of the multichain have connectivity $c=0$.
In the same fashion, we denote by $Z^{0}(\NC^{D}(p,q),m)$ the number of multichains of $\annularncDpq$ of length $m-2$ such that all elements have connectivity $c=0$. The following lemma provides a formula for $Z^{0}(\NC^{D}(p,q),m)$.

\begin{lemma}\label{lem:zero.connectivity.multichains}
  Let $p$ and $q$ be positive integers. Then 
  $$Z^{0}(\NC^{D}(p,q),m) = \frac{1}{m} Z^{0}(\NC^{B}(p,q),m)=
  \frac{1}{m}
  \binom{mp}{p}
  \binom{mq}{q}.$$
\end{lemma}

\begin{proof}
      Every zero-connectivity element $w$ in the interval $\interpq_{B}$ is the product $w=u \cdot v$ of an element $u \in [1,\wpp]$ and an element $v \in [1,w_{q}]$.
      By the characterization of $\annularncDpq$ recalled in \Cref{sec:annular}, such an element $w$ lies in $\interpq_{D}$ if and only if either both $u$ and $v$ have no zero-blocks, or both have a zero-block.
      Given a multichain $\mathcal{C} \in \MC_{p}(m,i)$ for some $0 \le i \le m-1$, we can complete it by any element of $\MC_{q}(m,i)$. So we have exactly
      \begin{equation}\label{eq:zero.connectivity}
        \sum_{i=0}^{m-1} |\MC_{p}(m,i)| \cdot |\MC_{q}(m,i)| =   \frac{1}{m}\binom{mp}{p}\binom{mq}{q}
      \end{equation}
       multichains of zero-connectivity elements, where the equality in \eqref{eq:zero.connectivity} follows from \Cref{lem:number.multichains.B}.
\end{proof}

We aim to show that every multichain of $\annularncBpq$ that contains an element $w$ such that $c(w) > 0$ is also a multichain of $\annularncDpq$.
We first prove a preliminary lemma.

\begin{lemma}\label{lem:interval.under.connectivity}
  Let $w \in \annularncBpq$ with $c(w)>0$. Consider an element $w' \in \annularncBpq$ such that
  $w' \le w$. Then $w' \in \annularncDpq$.
\end{lemma}

\begin{proof}
   If $w \in \NC^{B}(p,q)$ has connectivity $c>0$, then
   $w$ has no zero-blocks (see~\cite[Proposition~3.4]{NICA.annular.noncrossing}).  Then also $w'$ has no zero-blocks. 
   By the characterization of $\annularncDpq$ recalled in \Cref{sec:annular}, every element of $\annularncBpq$ with no zero-blocks lies in $\annularncDpq$, so $w' \in \annularncDpq$.
\end{proof}

 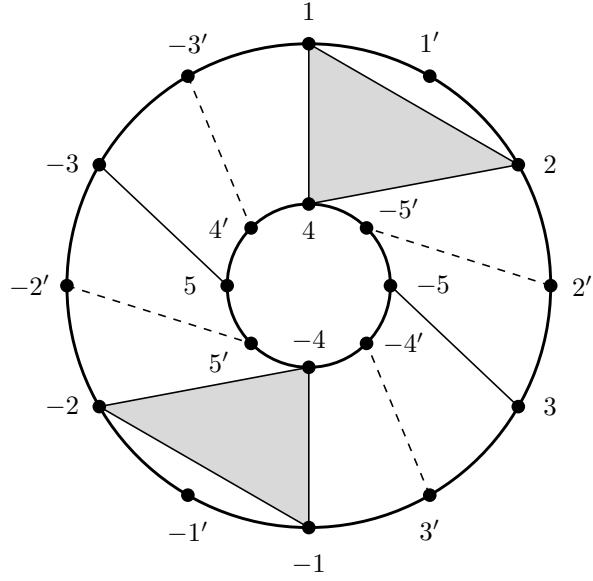
\begin{figure}[t]
\centering
\begin{tikzpicture}[scale=0.8]

\def\R{4}
\def\r{1.35}

\draw[line width=1.1pt] (0,0) circle (\R);
\draw[line width=1.1pt] (0,0) circle (\r);

\coordinate (P1)  at (90:\R);
\coordinate (P2)  at (30:\R);
\coordinate (P3)  at (-30:\R);
\coordinate (Pm1) at (-90:\R);
\coordinate (Pm2) at (-150:\R);
\coordinate (Pm3) at (150:\R);
\coordinate (P1')  at (60:\R);
\coordinate (P2')  at (0:\R);
\coordinate (P3')  at (-60:\R);
\coordinate (Pm1') at (-120:\R);
\coordinate (Pm2') at (-180:\R);
\coordinate (Pm3') at (120:\R);

\coordinate (P4)  at (90:\r);
\coordinate (Pm5) at (0:\r);
\coordinate (Pm4) at (-90:\r);
\coordinate (P5)  at (180:\r);
\coordinate (P4')  at (135:\r);
\coordinate (Pm5') at (45:\r);
\coordinate (Pm4') at (-45:\r);
\coordinate (P5')  at (-135:\r);

\fill[gray!30] (P1) -- (P2) -- (P4) -- cycle;
\fill[gray!30] (Pm1) -- (Pm2) -- (Pm4) -- cycle;

\draw[semithick] (P1)--(P2);
\draw[semithick] (P2)--(P4);
\draw[semithick] (P4)--(P1);

\draw[semithick] (Pm1)--(Pm2);
\draw[semithick] (Pm2)--(Pm4);
\draw[semithick] (Pm4)--(Pm1);

\draw[semithick] (P3)--(Pm5);
\draw[semithick] (Pm3)--(P5);

\draw[semithick, dashed] (P2')--(Pm5');
\draw[semithick, dashed] (Pm2')--(P5');
\draw[semithick, dashed] (P3')--(Pm4');
\draw[semithick, dashed] (Pm3')--(P4');

\foreach \x in {P1,P2,P3,Pm1,Pm2,Pm3,P4,Pm5,Pm4,P5}
{
    \fill (\x) circle (3.2pt);
}
\foreach \x in {P1',P2',P3',Pm1',Pm2',Pm3',P4',Pm5',Pm4',P5'}
{
    \fill (\x) circle (3.2pt);
}

\node[yshift=12pt] at (P1) {$1$};
\node[xshift=12pt] at (P2) {$2$};
\node[xshift=12pt] at (P3) {$3$};

\node[yshift=-14pt] at (Pm1) {$-1$};
\node[xshift=-14pt] at (Pm2) {$-2$};
\node[xshift=-14pt] at (Pm3) {$-3$};

\node[yshift=-10pt] at (P4) {$4$};
\node[xshift=16pt] at (Pm5) {$-5$};
\node[yshift=10pt] at (Pm4) {$-4$};
\node[xshift=-14pt] at (P5) {$5$};

\node[yshift=13pt] at (P1') {$1'$};
\node[xshift=12pt] at (P2') {$2'$};
\node[yshift=-13pt] at (P3') {$3'$};

\node[yshift=-14pt] at (Pm1') {$-1'$};
\node[xshift=-14pt] at (Pm2') {$-2'$};
\node[yshift=12pt] at (Pm3') {$-3'$};

\node[xshift=-12pt] at (P4') {$4'$};
\node[xshift=12pt,yshift=7pt] at (Pm5') {$-5'$};
\node[xshift=14pt] at (Pm4') {$-4'$};
\node[xshift=-12pt,yshift=-7pt] at (P5') {$5'$};

\end{tikzpicture}
\caption{The annular noncrossing partition of $((1,2,4))((3,-5))$ and its Kreweras complement $((2,-5))((3,-4))$ (dashed).}
\label{fig:annular.example}
\end{figure}

We describe how to compute the Kreweras complement $K_1^{\wpq}(w)$ for an element $w$ with $c(w)>0$, following \cite[Remark~6.7]{MingoNicaAnnular}.
As before, we add on the outer circle the points $1',\dots, p',-1',\dots,-p'$, where the point $i'$ is obtained with a clockwise rotation of $\frac{\pi}{2p}$ radians of the point $i \in \{1,\dots, p,-1,\dots,-p\}$.
In the same way, we also add on the inner circle the points $(p+1)',\dots, n',-(p+1)',\dots,-n'$, where the point $j'$ is obtained with a counterclockwise rotation of $\frac{\pi}{2q}$ radians of the point $j \in \{p+1,\dots, n,-(p+1),\dots,-n\}$.
Then $K_{1}^{\wpq}(w)$ corresponds to the coarsest partition of the set $\{1',\dots, n',-1',\dots,-n'\}$ such that no block of $w$ crosses a block of $K_{1}^{\wpq}(w)$.

We observe that $c\left(K^{\wpq}_{1}(w)\right)=c(w)$.
Indeed, the element $w$ sends exactly $2c(w)$ outer vertices to the inner circle, and exactly $2c(w)$ inner vertices to the outer circle,
since each of the $2c(w)$ cycles meeting both circles contains exactly one step in each direction
\cite[Definition~2.8 and Proposition~2.9]{NICA.annular.noncrossing}.
For each inner vertex $j$ sent by $w$ to the outer circle, let $i$ be the unique outer vertex satisfying $\wpq(i)=w(j)$. Then
\[
K_1^{\wpq}(w)(i)=w^{-1}(\wpq(i))=j.
\]
Conversely, every outer-to-inner step $i\mapsto j$ of the complement arises this way, since $w(j)=\wpq(i)$. This gives a bijection between the inner-to-outer steps of $w$ and the outer-to-inner steps of $K_1^{\wpq}(w)$. Consequently,
$2c(K_1^{\wpq}(w))=2c(w)$, proving the equality.
For example, if $w=((1,2,4))((3,-5))$, we see in Figure \ref{fig:annular.example} that $K_{1}^{\wpq}(w)=((2,-5))((3,-4))$; both elements $w$ and $K_{1}^{\wpq}(w)$ have connectivity $2$.

We now show that if $w \in \annularncBpq$ and $c(w)>0$, then every multichain containing $w$ is also a multichain of $\annularncDpq$.

\begin{lemma}\label{lem:connectivity.multichains}
  Let $\mathcal{C} \colon x_{1} \le \dots \le x_{m-1}$ be a multichain of $\annularncBpq$.
  Suppose that there exists $i \in \{1, \dots, m-1 \}$ such that $c(x_{i})>0$.
  Then $\mathcal{C}$ is also a multichain of $\annularncDpq$.
\end{lemma}

\begin{proof}
      By Lemma \ref{lem:interval.under.connectivity}, we know that $x_{j} \in \annularncDpq$, for every $j \le i$. 
      We have observed that $c(x_{i})=c(K_{1}^{\wpq}(x_{i}))>0$.
      By Remark \ref{rmk:interval.isomorphism}, the map $y\mapsto x_{i}y$ is a poset isomorphism from $[1,K^{\wpq}_{1}(x_{i})]_{B}$ to $[x_{i},\wpq]_{B}$, and likewise in type $D$.
      Since $c(K_{1}^{\wpq}(x_{i}))>0$, Lemma \ref{lem:interval.under.connectivity} gives
      $$[1,K^{\wpq}_{1}(x_{i})]_{B}=[1,K^{\wpq}_{1}(x_{i})]_{D}.$$
      Applying the same map $y\mapsto x_{i}y$ to both intervals, we obtain
      $$[x_{i},\wpq]_{B}=[x_{i},\wpq]_{D}.$$
      Hence $x_{j} \in [x_{i},\wpq]_{D}$ for all $j \ge i$.
      Since $\annularncDpq$ is an induced subposet of $\annularncBpq$, it follows that $\mathcal{C}$ is a multichain of $\annularncDpq$.
\end{proof}

We finish the section with the proof of \Cref{thm:zeta.polynomial.D}.

\begin{proof}[Proof of \Cref{thm:zeta.polynomial.D}]
    Let $x_{1} \le \dots \le x_{m-1}$ be a multichain of $\annularncDpq$.
    Either $c(x_{i})=0$ for all $i$, or $c(x_{i})>0$ for some $i$.
    By \Cref{lem:zero.connectivity.multichains}, there are $\frac{1}{m}\binom{mp}{p}\binom{mq}{q}$ multichains of the first kind.
    By \Cref{lem:connectivity.multichains}, the multichains of the second kind are exactly the multichains of $\annularncBpq$ that contain an element of positive connectivity, which are counted by the second summand of \eqref{eq:zeta.annular.B}.
    Hence
    \begin{align*}
    Z(\NC^{D}(p,q),m)
    &= \frac{1}{m}
    \binom{mp}{p}
    \binom{mq}{q}
    +
    \sum_{c=1}^{p}
    2c
    \binom{mp}{p-c}
    \binom{mq}{q+c} \\
    &= Z(\NC^{B}(p,q),m) - \left(1-\frac{1}{m}\right)\binom{mp}{p}\binom{mq}{q}.
    \end{align*}
    By \Cref{cor:zeta.poly.B.closed}, we obtain the closed formula
    \[ Z(\NC^{D}(p,q),m) = \frac{p+q+2pq(m-1)}{(p+q)m} \binom{mp}{p} \binom{mq}{q}. \qedhere \]
\end{proof}

From \Cref{thm:zeta.polynomial.D}, we get the following enumerative results for $\annularncDpq$.

\begin{corollary}\label{cor:enumerative.results}
  Let $p$ and $q$ be two positive integers, and let $P = \annularncDpq$. Then:
  \begin{enumerate}[label=(\roman*)]
      \item The number of elements of $P$ is 
      $$|P|= \frac{p+q+2pq}{2(p+q)} \binom{2p}{p} \binom{2q}{q};$$
      \item The M\"obius function of $P$ between its minimum and maximum is 
      $$\mu_{P}\left(1,\wpq \right)=(-1)^{p+q} \frac{4pq-p-q}{4(p+q)} \binom{2p}{p} \binom{2q}{q};$$
      \item The number of maximal chains in $P$ is 
      $$ \frac{2pq}{p+q} \binom{p+q}{p} p^{p} q^{q}.$$
  \end{enumerate}
\end{corollary}

\begin{proof}
      All the results follow from \Cref{thm:zeta.polynomial.D} and \cite[Proposition 3.12.1]{Stanley_2011}. 
\end{proof}

\section{Exceptional types}
\label{sec:exceptional-cases}

In this section, we prove our results for the exceptional types.
If $W$ is of type $I_2(m)$, then the interval $[1,w]$ below a quasi-Coxeter element $w$ has rank two.
Every bounded poset of rank two is EL-shellable, and by \Cref{rmk:rising.chain} a total order of $\refl$ makes the natural labeling of $[1,w]$ an EL-labeling if and only if it is $w$-compatible.
So \Cref{thm:main-shellability,thm:main-shellability-refined} are immediate in this case.

It remains to consider the types $H_3$, $F_4$, $H_4$, $E_6$, $E_7$, and $E_8$.
Their proper quasi-Coxeter elements form $27$ conjugacy classes, which we denote as in \cite{carter1972conjugacy,baumeister2023interval-part2}:
\[
\begin{gathered}
H_3(a_1),\ H_3(a_2),\quad F_4(a_1),\quad H_4(a_1),\ \dots,\ H_4(a_{10}),\\
E_6(a_1),\ E_6(a_2),\quad E_7(a_1),\ \dots,\ E_7(a_4),\quad E_8(a_1),\ \dots,\ E_8(a_8).
\end{gathered}
\]
We study the corresponding intervals by computer.

Our results are summarized in \Cref{tab:exceptional-summary}, together with those of \Cref{sec:interval-monoid} on interval monoids.
In \Cref{subsec:compatible-exceptional} we determine the classes that admit a $w$-compatible order, and in \Cref{subsec:every-compatible} we show that for these classes, and for the Coxeter elements, every $w$-compatible order makes the natural labeling an EL-labeling.
In \Cref{subsec:non-natural} we find EL-labelings of a different kind for the remaining classes of type $H_4$.
Finally, \Cref{subsec:zeta-exceptional} records the zeta polynomials.

\begin{table}[htbp]
\centering
\footnotesize
\renewcommand{\arraystretch}{1.12}
\setlength{\tabcolsep}{4pt}
\begin{tabular}{l|r|r|r|c|c|c|c}
\hline
class & order & $|[1,w]|$ & $|\mu(1,w)|$ & lattice
 & \begin{tabular}[b]{@{}c@{}}natural\\ EL\end{tabular}
 & \begin{tabular}[b]{@{}c@{}}EL-\\ shellable\end{tabular}
 & \begin{tabular}[b]{@{}c@{}}$M_w$\\ cancellative\end{tabular} \\
\hline
$H_3(a_1)$ & $10$ & $32$ & $21$ & yes & yes & yes & yes \\
$H_3(a_2)$ & $6$ & $32$ & $25$ & yes & yes & yes & yes \\
\hline
$F_4(a_1)$ & $6$ & $128$ & $125$ & no & yes & yes & ? \\
\hline
$H_4(a_1)$ & $30$ & $280$ & $232$ & yes & yes & yes & yes \\
$H_4(a_2)$ & $20$ & $286$ & $323$ & no & no & yes & ? \\
$H_4(a_3)$ & $20$ & $286$ & $323$ & no & no & yes & ? \\
$H_4(a_4)$ & $15$ & $323$ & $455$ & no & no & yes & ? \\
$H_4(a_5)$ & $15$ & $323$ & $455$ & no & no & yes & ? \\
$H_4(a_6)$ & $12$ & $294$ & $363$ & no & no & yes & ? \\
$H_4(a_7)$ & $10$ & $332$ & $554$ & no & no & yes & ? \\
$H_4(a_8)$ & $10$ & $432$ & $729$ & no & no & yes & ? \\
$H_4(a_9)$ & $10$ & $432$ & $729$ & no & no & yes & ? \\
$H_4(a_{10})$ & $6$ & $404$ & $941$ & no & no & yes & ? \\
\hline
$E_6(a_1)$ & $9$ & $1001$ & $728$ & no & yes & yes & no \\
$E_6(a_2)$ & $6$ & $1152$ & $1125$ & no & yes & yes & no \\
\hline
$E_7(a_1)$ & $14$ & $5280$ & $4719$ & no & yes & yes & no \\
$E_7(a_2)$ & $12$ & $5700$ & $5610$ & no & yes & yes & no \\
$E_7(a_3)$ & $30$ & $6688$ & $8970$ & no & no & ? & no \\
$E_7(a_4)$ & $6$ & $8192$ & $15\,625$ & no & no & ? & no \\
\hline
$E_8(a_1)$ & $24$ & $33\,592$ & $37\,145$ & no & yes & yes & no \\
$E_8(a_2)$ & $20$ & $38\,896$ & $50\,388$ & no & no & ? & no \\
$E_8(a_3)$ & $12$ & $46\,305$ & $71\,874$ & no & no & ? & no \\
$E_8(a_4)$ & $18$ & $46\,475$ & $83\,300$ & no & no & ? & no \\
$E_8(a_5)$ & $15$ & $52\,003$ & $105\,105$ & no & no & ? & no \\
$E_8(a_6)$ & $10$ & $64\,827$ & $186\,624$ & no & no & ? & no \\
$E_8(a_7)$ & $12$ & $67\,200$ & $202\,125$ & no & no & ? & no \\
$E_8(a_8)$ & $6$ & $102\,014$ & $591\,251$ & no & no & ? & no \\
\hline
\end{tabular}
\caption{The proper quasi-Coxeter classes of the exceptional types.
The columns give the order of $w$; the cardinality of $[1,w]$ and $|\mu(1,w)|$, both obtained from \Cref{tab:exceptional-zeta.Weyl,tab:exceptional-zeta.H}; whether $[1,w]$ is a lattice \cite[Theorem~2.12]{baumeister2023interval-part2}; whether some total order of $\refl$ makes the natural labeling an EL-labeling, which happens exactly when a $w$-compatible order exists (\Cref{thm:exceptional-compatible-exists} and \Cref{subsec:every-compatible}); whether $[1,w]$ is EL-shellable (\Cref{cor:EL-not-E7-E8}); and whether the interval monoid $M_w$ is cancellative (\Cref{sec:interval-monoid}).
A question mark means that the answer is not known.}
\label{tab:exceptional-summary}
\end{table}

\subsection{Compatible orders}
\label{subsec:compatible-exceptional}

We first determine which proper quasi-Coxeter elements admit a $w$-compatible order.

\begin{theorem}\label{thm:exceptional-compatible-exists}
    Let $w$ be a proper quasi-Coxeter element of a Coxeter group of type $H_3$, $F_4$, $H_4$, $E_6$, $E_7$, or $E_8$.
    Then a $w$-compatible total order of $\refl$ exists if and only if $w$ belongs to one of the nine classes
    \[
    H_3(a_1),\ H_3(a_2),\ F_4(a_1),\ H_4(a_1),\ E_6(a_1),\ E_6(a_2),\ E_7(a_1),\ E_7(a_2),\ E_8(a_1).
    \]
\end{theorem}

\begin{proof}
    For each of the $27$ classes, we checked by computer whether a $w$-compatible order exists; it exists exactly for the nine classes in the statement.
\end{proof}

For the other eighteen classes, no total order of $\refl$ is $w$-compatible, so none makes the natural labeling an EL-labeling.
Since such classes exist in each of the types $H_4$, $E_7$, and $E_8$, this proves the second part of \Cref{thm:main-shellability-refined}.

\subsection{Every compatible order yields an EL-labeling}
\label{subsec:every-compatible}

We now show that, as in the infinite families (\Cref{thm:compatible-EL-ABD}), every $w$-compatible order makes the natural labeling an EL-labeling, for Coxeter elements and proper quasi-Coxeter elements alike; for Coxeter elements, this is Williams's conjecture (see \cite[Section~7.3]{defant2025permutahedron}).
For the Coxeter elements and for the nine classes of \Cref{thm:exceptional-compatible-exists}, we checked by computer that every $w$-compatible total order of $\refl$ makes the natural labeling of $[1,w]$ an EL-labeling.
For the other eighteen classes there is nothing to check, because no $w$-compatible order exists.

Since every total order that makes the natural labeling an EL-labeling is $w$-compatible, it follows that, if $w$ is a proper quasi-Coxeter element of type $H_3$, $F_4$, $H_4$, $E_6$, $E_7$, or $E_8$, then some total order of $\refl$ makes the natural labeling of $[1,w]$ an EL-labeling if and only if $w$ belongs to one of the nine classes of \Cref{thm:exceptional-compatible-exists}, and in this case every $w$-compatible order does.

All the proper quasi-Coxeter elements of types $H_3$, $F_4$, and $E_6$ belong to the classes of \Cref{thm:exceptional-compatible-exists}, and the Coxeter elements are covered by \Cref{thm:EL-labeling.for.Coxeter.elements}.
So \Cref{thm:main-shellability} holds for these three types.
Moreover, the computation above, together with \Cref{thm:compatible-EL-ABD} and the dihedral case, shows that every $w$-compatible order yields an EL-labeling for every quasi-Coxeter element $w$ of a finite Coxeter group; if $W$ is reducible, then $[1,w]$ is the product of the intervals below the components of $w$, and we argue as in the proof of \Cref{thm:compatible-EL-ABD}.
This completes the proof of \Cref{thm:main-shellability-refined}; in particular, Williams's conjecture holds for all finite Coxeter groups.

\begin{remark}[Thomas's reduction in the exceptional types]\label{rem:thomas-exceptional}
    By Thomas's argument in \Cref{rem:thomas.proof}, to show that every $w$-compatible order yields an EL-labeling of $[1,w]$, it is enough to show that, for every $x \le w$ with $\operatorname{rank}(x) \le m$ (where $m$ is as in \Cref{rem:thomas.proof}), every $x$-compatible order yields an EL-labeling of $[1,x]$.
    In the exceptional types, this reduction does not shorten the computation much.
    Indeed, if $n$ is the rank of $W$, then $m = n$ for all the Coxeter elements and for the classes $H_3(a_1)$, $H_3(a_2)$, $H_4(a_1)$, and $E_8(a_1)$, so that the reduction does not help; $m = n-1$ for $F_4(a_1)$, $E_6(a_1)$, $E_7(a_1)$, and $E_7(a_2)$; and $m = n-2$ for $E_6(a_2)$.
\end{remark}

\subsection{EL-labelings that are not natural}
\label{subsec:non-natural}

For the eighteen classes that do not appear in \Cref{thm:exceptional-compatible-exists}, EL-shellability can only come from edge labelings other than the natural one.
We find such labelings for all the classes of type $H_4$.

\begin{theorem}\label{thm:H4-EL}
    Let $w$ be a quasi-Coxeter element of a Coxeter group of type $H_4$.
    Then $[1,w]$ is EL-shellable.
\end{theorem}

\begin{proof}
    For the Coxeter elements this is \Cref{thm:EL-labeling.for.Coxeter.elements}, and for $H_4(a_1)$ it follows from \Cref{thm:exceptional-compatible-exists} and \Cref{subsec:every-compatible}.
    For each of the other nine classes, we found by computer an EL-labeling of $[1,w]$.
\end{proof}

\begin{corollary}\label{cor:EL-not-E7-E8}
    Let $w$ be a quasi-Coxeter element of a finite irreducible Coxeter group $W$.
    If $W$ is not of type $E_7$ or $E_8$, then $[1,w]$ is EL-shellable.
\end{corollary}

\begin{proof}
    For the infinite families and the dihedral types, see \Cref{sec:shellability-type-D} and the beginning of this section.
    In the exceptional types, the Coxeter elements are covered by \Cref{thm:EL-labeling.for.Coxeter.elements}, and the proper quasi-Coxeter elements by \Cref{thm:exceptional-compatible-exists}, \Cref{subsec:every-compatible}, and \Cref{thm:H4-EL}.
\end{proof}

In particular, the conclusion of \Cref{thm:main-shellability} also holds for $W$ of type $H_4$.
For the nine classes $E_7(a_3)$, $E_7(a_4)$, $E_8(a_2), \dots, E_8(a_8)$ we do not know whether $[1,w]$ is EL-shellable, although \Cref{conj:EL} predicts that it is.
These intervals are much larger than those of type $H_4$ (\Cref{tab:exceptional-summary}), and they are beyond the reach of our computations.

\subsection{Zeta polynomials}
\label{subsec:zeta-exceptional}

For a conjugacy class $\Gamma$, we write $Z_\Gamma(m)$ for the zeta polynomial of $[1,w]$ with $w \in \Gamma$.
If $c_k$ denotes the number of chains $x_1 < x_2 < \dots < x_k$ with $k$ elements in $[1,w]$, then it is proven in \cite[Proposition 3.12.1]{Stanley_2011} that
\[
Z_\Gamma(m) = \sum_{k \ge 1} c_k \binom{m-2}{k-1},
\]
because the multichains with $m-1$ elements and $k$ distinct values correspond bijectively to the pairs formed by a chain with $k$ elements and a composition of $m-1$ into $k$ parts, which records the multiplicities.
We computed the numbers $c_k$, and hence the polynomials $Z_\Gamma(m)$, by computer; they are listed in \Cref{tab:exceptional-zeta.Weyl,tab:exceptional-zeta.H}.

\begin{table}[p]
\centering
\scriptsize
\setlength{\tabcolsep}{4pt}
\begin{tabular}{l|r|l}
\hline
class & $|[1,w]|$ & $Z_\Gamma(m)$\rule[-1.2ex]{0pt}{3.4ex} \\
\hline
$F_4(a_1)$ & $128$ & $\displaystyle \rule{0pt}{3.2ex}27m^4-54m^3+36m^2-8m\rule[-1.5ex]{0pt}{0pt}$ \\
\hline
$E_6(a_1)$ & $1001$ & $\displaystyle \rule{0pt}{4.3ex}\frac{6561}{80}m^6-\frac{3645}{16}m^5+\frac{3969}{16}m^4-\frac{2115}{16}m^3+\frac{1377}{40}m^2-\frac{7}{2}m\rule[-2.6ex]{0pt}{0pt}$ \\
$E_6(a_2)$ & $1152$ & $\displaystyle \rule{0pt}{3.2ex}108m^6-324m^5+387m^4-230m^3+68m^2-8m\rule[-1.5ex]{0pt}{0pt}$ \\
\hline
$E_7(a_1)$ & $5280$ & $\displaystyle \begin{aligned}[t] &\rule{0pt}{4.3ex}\frac{117649}{360}m^7-\frac{16807}{15}m^6+\frac{112847}{72}m^5-\frac{3430}{3}m^4 \\[0.6ex] &\quad +\frac{41209}{90}m^3-\frac{476}{5}m^2+8m\rule[-2.6ex]{0pt}{0pt} \end{aligned}$ \\
$E_7(a_2)$ & $5700$ & $\displaystyle \begin{aligned}[t] &\rule{0pt}{4.3ex}\frac{2592}{7}m^7-1296m^6+1854m^5-1386m^4 \\[0.6ex] &\quad +\frac{1141}{2}m^3-\frac{245}{2}m^2+\frac{75}{7}m\rule[-2.6ex]{0pt}{0pt} \end{aligned}$ \\
$E_7(a_3)$ & $6688$ & $\displaystyle \begin{aligned}[t] &\rule{0pt}{4.3ex}\frac{28125}{56}m^7-1875m^6+\frac{23125}{8}m^5-\frac{14125}{6}m^4 \\[0.6ex] &\quad +\frac{2135}{2}m^3-\frac{766}{3}m^2+\frac{176}{7}m\rule[-2.6ex]{0pt}{0pt} \end{aligned}$ \\
$E_7(a_4)$ & $8192$ & $\displaystyle \rule{0pt}{3.2ex}729m^7-2916m^6+4860m^5-4320m^4+2160m^3-576m^2+64m\rule[-1.5ex]{0pt}{0pt}$ \\
\hline
$E_8(a_1)$ & $33\,592$ & $\displaystyle \begin{aligned}[t] &\rule{0pt}{4.3ex}\frac{55296}{35}m^8-\frac{32256}{5}m^7+\frac{55296}{5}m^6-10304m^5 \\[0.6ex] &\quad +\frac{84376}{15}m^4-\frac{8974}{5}m^3+\frac{32429}{105}m^2-22m\rule[-2.6ex]{0pt}{0pt} \end{aligned}$ \\
$E_8(a_2)$ & $38\,896$ & $\displaystyle \begin{aligned}[t] &\rule{0pt}{4.3ex}\frac{125000}{63}m^8-\frac{25000}{3}m^7+\frac{132500}{9}m^6-\frac{42500}{3}m^5 \\[0.6ex] &\quad +\frac{144175}{18}m^4-\frac{5315}{2}m^3+\frac{3349}{7}m^2-36m\rule[-2.6ex]{0pt}{0pt} \end{aligned}$ \\
$E_8(a_3)$ & $46\,305$ & $\displaystyle \begin{aligned}[t] &\rule{0pt}{3.2ex}2592m^8-11232m^7+20520m^6-20472m^5 \\[0.6ex] &\quad +12038m^4-4170m^3+\frac{1575}{2}m^2-\frac{125}{2}m\rule[-2.6ex]{0pt}{0pt} \end{aligned}$ \\
$E_8(a_4)$ & $46\,475$ & $\displaystyle \begin{aligned}[t] &\rule{0pt}{4.3ex}\frac{177147}{64}m^8-\frac{98415}{8}m^7+\frac{741393}{32}m^6-\frac{383211}{16}m^5 \\[0.6ex] &\quad +\frac{938439}{64}m^4-\frac{84951}{16}m^3+\frac{16815}{16}m^2-\frac{175}{2}m\rule[-2.6ex]{0pt}{0pt} \end{aligned}$ \\
$E_8(a_5)$ & $52\,003$ & $\displaystyle \begin{aligned}[t] &\rule{0pt}{4.3ex}\frac{421875}{128}m^8-\frac{478125}{32}m^7+\frac{1839375}{64}m^6-\frac{243125}{8}m^5 \\[0.6ex] &\quad +\frac{7322225}{384}m^4-\frac{227035}{32}m^3+\frac{139063}{96}m^2-\frac{1001}{8}m\rule[-2.6ex]{0pt}{0pt} \end{aligned}$ \\
$E_8(a_6)$ & $64\,827$ & $\displaystyle \begin{aligned}[t] &\rule{0pt}{4.3ex}\frac{78125}{16}m^8-\frac{46875}{2}m^7+\frac{384375}{8}m^6-\frac{218125}{4}m^5 \\[0.6ex] &\quad +\frac{592125}{16}m^4-\frac{60075}{4}m^3+3375m^2-324m\rule[-2.6ex]{0pt}{0pt} \end{aligned}$ \\
$E_8(a_7)$ & $67\,200$ & $\displaystyle \begin{aligned}[t] &\rule{0pt}{3.2ex}5184m^8-25056m^7+51732m^6-59136m^5 \\[0.6ex] &\quad +40415m^4-16510m^3+3732m^2-360m\rule[-1.5ex]{0pt}{0pt} \end{aligned}$ \\
$E_8(a_8)$ & $102\,014$ & $\displaystyle \begin{aligned}[t] &\rule{0pt}{3.2ex}10935m^8-58320m^7+134136m^6-172530m^5 \\[0.6ex] &\quad +134055m^4-62910m^3+16500m^2-1865m\rule[-1.5ex]{0pt}{0pt} \end{aligned}$ \\
\hline
\end{tabular}
\caption{Zeta polynomials $Z_\Gamma(m)$ of the intervals $[1,w]$ for the proper quasi-Coxeter classes $\Gamma$ of types $F_4$, $E_6$, $E_7$, and $E_8$.}
\label{tab:exceptional-zeta.Weyl}
\end{table}

\begin{table}[tbp]
\centering
\scriptsize
\setlength{\tabcolsep}{4pt}
\begin{tabular}{l|r|l}
\hline
class & $|[1,w]|$ & $Z_\Gamma(m)$\rule[-1.2ex]{0pt}{3.4ex} \\
\hline
$H_3(a_1)$ & $32$ & $\displaystyle \rule{0pt}{4.3ex}\frac{25}{3}m^3-10m^2+\frac{8}{3}m\rule[-2.6ex]{0pt}{0pt}$ \\
$H_3(a_2)$ & $32$ & $\displaystyle \rule{0pt}{3.2ex}9m^3-12m^2+4m\rule[-1.5ex]{0pt}{0pt}$ \\
\hline
$H_4(a_1)$ & $280$ & $\displaystyle \rule{0pt}{4.3ex}\frac{225}{4}m^4-105m^3+\frac{241}{4}m^2-\frac{21}{2}m\rule[-2.6ex]{0pt}{0pt}$ \\
$H_4(a_2)$ & $286$ & $\displaystyle \rule{0pt}{4.3ex}\frac{200}{3}m^4-140m^3+\frac{286}{3}m^2-21m\rule[-2.6ex]{0pt}{0pt}$ \\
$H_4(a_3)$ & $286$ & $\displaystyle \rule{0pt}{4.3ex}\frac{200}{3}m^4-140m^3+\frac{286}{3}m^2-21m\rule[-2.6ex]{0pt}{0pt}$ \\
$H_4(a_4)$ & $323$ & $\displaystyle \rule{0pt}{4.3ex}\frac{675}{8}m^4-\frac{765}{4}m^3+\frac{1149}{8}m^2-\frac{143}{4}m\rule[-2.6ex]{0pt}{0pt}$ \\
$H_4(a_5)$ & $323$ & $\displaystyle \rule{0pt}{4.3ex}\frac{675}{8}m^4-\frac{765}{4}m^3+\frac{1149}{8}m^2-\frac{143}{4}m\rule[-2.6ex]{0pt}{0pt}$ \\
$H_4(a_6)$ & $294$ & $\displaystyle \rule{0pt}{3.2ex}72m^4-156m^3+110m^2-25m\rule[-1.5ex]{0pt}{0pt}$ \\
$H_4(a_7)$ & $332$ & $\displaystyle \rule{0pt}{4.3ex}\frac{375}{4}m^4-225m^3+\frac{735}{4}m^2-\frac{103}{2}m\rule[-2.6ex]{0pt}{0pt}$ \\
$H_4(a_8)$ & $432$ & $\displaystyle \rule{0pt}{3.2ex}125m^4-300m^3+240m^2-64m\rule[-1.5ex]{0pt}{0pt}$ \\
$H_4(a_9)$ & $432$ & $\displaystyle \rule{0pt}{3.2ex}125m^4-300m^3+240m^2-64m\rule[-1.5ex]{0pt}{0pt}$ \\
$H_4(a_{10})$ & $404$ & $\displaystyle \rule{0pt}{3.2ex}135m^4-360m^3+336m^2-110m\rule[-1.5ex]{0pt}{0pt}$ \\
\hline
\end{tabular}
\caption{Zeta polynomials $Z_\Gamma(m)$ of the intervals $[1,w]$ for the proper quasi-Coxeter classes $\Gamma$ of types $H_3$ and $H_4$.}
\label{tab:exceptional-zeta.H}
\end{table}

\section{Interval monoids}
\label{sec:interval-monoid}

In this section we prove \Cref{thm:main-noncancellative}, using the facts on interval monoids and their parabolic submonoids recalled in \Cref{subsec:interval-groups}.
By \Cref{prop:parabolic-submonoid}, the proof of \Cref{thm:main-noncancellative} has two steps.
First, in \Cref{subsec:D5}, we find an explicit failure of cancellation for the proper quasi-Coxeter element $w_{3,2}$ of $D_5$.
Then, in \Cref{subsec:proof-a}, we show that every proper quasi-Coxeter element $w$ of type $D_n$ with $n \geq 5$, or of type $E_6$, $E_7$, or $E_8$, lies above an element $x$ whose interval $[1,x]$ is isomorphic to $[1,w_{3,2}]$ as a labeled poset.
The monoid $M_x$ is then isomorphic to $M_{w_{3,2}}$, so it is not cancellative, and by \Cref{prop:parabolic-submonoid} neither is $M_w$.

\subsection{A counterexample in type \texorpdfstring{$D_5$}{D5}}
\label{subsec:D5}

Let $W$ be of type $D_5$ and let $w = w_{3,2} = [1,2,3][4,5]$.
We give names to seven reflections:
\begin{gather*}
  a = ((1,-2)),\quad b = ((4,5)),\quad c = ((3,4)),\quad d = ((4,-5)), \\
  e = ((3,5)),\quad f = ((3,-5)),\quad g = ((1,-4));
\end{gather*}
and we consider the words
\[
  X = abcad, \qquad P = Xe = abcade, \qquad Q = fX = fabcad.
\]

\begin{proposition}\label{prop:D5-counterexample}
    In $M_{w_{3,2}}$ we have $Pg = Qg$, but $P \neq Q$.
    Hence $M_{w_{3,2}}$ is neither left nor right cancellative.
    Since $P$ and $Q$ have the same image in the interval group $G_{w_{3,2}}$, the natural map $M_{w_{3,2}} \to G_{w_{3,2}}$ is not injective.
\end{proposition}

\begin{table}[ht]
\centering
\small
\renewcommand{\arraystretch}{1.15}
\begin{tabular}{c|c}
\hline
pair $st$ & $y = st$ \\
\hline
$bc$ & $((3,5,4))$ \\
$ca$ & $((1,-2))((3,4))$ \\
$ac$ & $((1,-2))((3,4))$ \\
$cd$ & $((3,4,-5))$ \\
$de$ & $((3,-4,5))$ \\
$ae$ & $((1,-2))((3,5))$ \\
\hline
\end{tabular}
\caption{The pairs of adjacent letters $st$ in the words of the class of $P$ to which no relation of $M_{w_{3,2}}$ applies, since $y \not\le w_{3,2}$.}
\label{tab:D5-blocked}
\end{table}

\begin{proof}
    The reflection $a$ commutes with $b$ and with $d$, and $ab$ and $ad$ lie below $w_{3,2}$, so the relations $ab = ba$ and $ad = da$ hold in $M_{w_{3,2}}$.
    Hence the class of $P$ contains the four words
    \[
        abcade, \qquad bacade, \qquad abcdae, \qquad bacdae.
    \]
    We claim that there are no other words in the class.
    By \Cref{lem:rank-two-presentation}, it is enough to look at the relations that apply to pairs of adjacent letters in these four words.
    The element $ab$ has no reduced factorizations other than $ab$ and $ba$, and similarly for $ad$.
    The remaining pairs of adjacent letters are $bc$, $ca$, $ac$, $cd$, $de$, and $ae$, and \Cref{tab:D5-blocked} shows that none of the corresponding products lies below $w_{3,2}$, so that no relation applies to them.
    This proves the claim.
    Since no word in the class of $P$ begins with $f$, we conclude that $P \neq Q$ in $M_{w_{3,2}}$.

    \Cref{tab:D5-chain} lists $23$ elementary relations that transform $Pg$ into $Qg$.
    In each step, two adjacent letters $st$ are replaced by another reduced factorization of $y = st$, and one checks that $y \le w_{3,2}$.
    Therefore, $Pg = Qg$ in $M_{w_{3,2}}$.

    For left cancellation, let $h = ((1,3))((4,5))$, so that $h w_{3,2} h = w_{3,2}^{-1}$, and hence $x \le w_{3,2}$ if and only if $h x^{-1} h \le w_{3,2}$.
    Therefore, the map $t_1 \cdots t_k \mapsto (h t_k h) \cdots (h t_1 h)$ is an anti-automorphism of $M_{w_{3,2}}$, which turns the equality $Pg = Qg$ into a failure of left cancellation.

    Finally, the equality $Pg = Qg$ implies $P = Q$ in the group $G_{w_{3,2}}$, while $P \neq Q$ in the monoid $M_{w_{3,2}}$, hence $M_{w_{3,2}} \to G_{w_{3,2}}$ is not injective.
\end{proof}

\begin{table}[tp]
\centering
\scriptsize
\renewcommand{\arraystretch}{1.12}
\setlength{\tabcolsep}{2.2pt}
\begin{tabular}{r|ccccccc|l}
\hline
 & \multicolumn{7}{c|}{word} & rank-$2$ element\\
\hline
$Pg$ & $((1,-2))$ & $((4,5))$ & $((3,4))$ & $((1,-2))$ & $((4,-5))$ & \underline{$((3,5))$} & \underline{$((1,-4))$} & $((1,-4))((3,5))$ \\
1 & $((1,-2))$ & $((4,5))$ & $((3,4))$ & $((1,-2))$ & \underline{$((4,-5))$} & \underline{$((1,-4))$} & $((3,5))$ & $((1,5,-4))$ \\
2 & $((1,-2))$ & $((4,5))$ & $((3,4))$ & \underline{$((1,-2))$} & \underline{$((1,-4))$} & $((1,5))$ & $((3,5))$ & $((1,-4,-2))$ \\
3 & $((1,-2))$ & $((4,5))$ & \underline{$((3,4))$} & \underline{$((2,4))$} & $((1,-2))$ & $((1,5))$ & $((3,5))$ & $((2,3,4))$ \\
4 & $((1,-2))$ & \underline{$((4,5))$} & \underline{$((2,3))$} & $((3,4))$ & $((1,-2))$ & $((1,5))$ & $((3,5))$ & $((2,3))((4,5))$ \\
5 & \underline{$((1,-2))$} & \underline{$((2,3))$} & $((4,5))$ & $((3,4))$ & $((1,-2))$ & $((1,5))$ & $((3,5))$ & $((1,-2,-3))$ \\
6 & $((1,-3))$ & $((1,-2))$ & $((4,5))$ & $((3,4))$ & \underline{$((1,-2))$} & \underline{$((1,5))$} & $((3,5))$ & $((1,5,-2))$ \\
7 & $((1,-3))$ & $((1,-2))$ & $((4,5))$ & \underline{$((3,4))$} & \underline{$((1,5))$} & $((2,-5))$ & $((3,5))$ & $((1,5))((3,4))$ \\
8 & $((1,-3))$ & $((1,-2))$ & \underline{$((4,5))$} & \underline{$((1,5))$} & $((3,4))$ & $((2,-5))$ & $((3,5))$ & $((1,4,5))$ \\
9 & $((1,-3))$ & \underline{$((1,-2))$} & \underline{$((1,5))$} & $((1,4))$ & $((3,4))$ & $((2,-5))$ & $((3,5))$ & $((1,5,-2))$ \\
10 & \underline{$((1,-3))$} & \underline{$((1,5))$} & $((2,-5))$ & $((1,4))$ & $((3,4))$ & $((2,-5))$ & $((3,5))$ & $((1,5,-3))$ \\
11 & $((3,-5))$ & $((1,-3))$ & $((2,-5))$ & $((1,4))$ & $((3,4))$ & \underline{$((2,-5))$} & \underline{$((3,5))$} & $((2,-5,-3))$ \\
12 & $((3,-5))$ & $((1,-3))$ & $((2,-5))$ & $((1,4))$ & \underline{$((3,4))$} & \underline{$((2,-3))$} & $((2,-5))$ & $((2,-4,-3))$ \\
13 & $((3,-5))$ & $((1,-3))$ & $((2,-5))$ & \underline{$((1,4))$} & \underline{$((2,-4))$} & $((3,4))$ & $((2,-5))$ & $((1,4,-2))$ \\
14 & $((3,-5))$ & $((1,-3))$ & \underline{$((2,-5))$} & \underline{$((1,-2))$} & $((1,4))$ & $((3,4))$ & $((2,-5))$ & $((1,5,-2))$ \\
15 & $((3,-5))$ & \underline{$((1,-3))$} & \underline{$((1,-2))$} & $((1,5))$ & $((1,4))$ & $((3,4))$ & $((2,-5))$ & $((1,-2,-3))$ \\
16 & $((3,-5))$ & $((1,-2))$ & $((2,3))$ & \underline{$((1,5))$} & \underline{$((1,4))$} & $((3,4))$ & $((2,-5))$ & $((1,4,5))$ \\
17 & $((3,-5))$ & $((1,-2))$ & \underline{$((2,3))$} & \underline{$((4,5))$} & $((1,5))$ & $((3,4))$ & $((2,-5))$ & $((2,3))((4,5))$ \\
18 & $((3,-5))$ & $((1,-2))$ & $((4,5))$ & $((2,3))$ & \underline{$((1,5))$} & \underline{$((3,4))$} & $((2,-5))$ & $((1,5))((3,4))$ \\
19 & $((3,-5))$ & $((1,-2))$ & $((4,5))$ & \underline{$((2,3))$} & \underline{$((3,4))$} & $((1,5))$ & $((2,-5))$ & $((2,3,4))$ \\
20 & $((3,-5))$ & $((1,-2))$ & $((4,5))$ & $((3,4))$ & $((2,4))$ & \underline{$((1,5))$} & \underline{$((2,-5))$} & $((1,5,-2))$ \\
21 & $((3,-5))$ & $((1,-2))$ & $((4,5))$ & $((3,4))$ & \underline{$((2,4))$} & \underline{$((1,-2))$} & $((1,5))$ & $((1,-4,-2))$ \\
22 & $((3,-5))$ & $((1,-2))$ & $((4,5))$ & $((3,4))$ & $((1,-2))$ & \underline{$((1,-4))$} & \underline{$((1,5))$} & $((1,5,-4))$ \\
$Qg$ & $((3,-5))$ & $((1,-2))$ & $((4,5))$ & $((3,4))$ & $((1,-2))$ & $((4,-5))$ & $((1,-4))$ &  \\
\hline
\end{tabular}
\caption{A chain of $23$ elementary relations of $M_{w_{3,2}}$ transforming $Pg$ into $Qg$ (\Cref{prop:D5-counterexample}). In each step, the underlined pair $st$ is replaced with the other factorization shown in the next row, and the last column records the element $y = st \le w_{3,2}$ of rank two.}
\label{tab:D5-chain}
\end{table}

\begin{remark}\label{rem:D5-minimality}
    The counterexample was found by a computer search.
    An exhaustive computation shows that this is the smallest possible counterexample to cancellativity in $M_{w_{3,2}}$.
\end{remark}

\subsection{Types \texorpdfstring{$D_n$}{Dn} with \texorpdfstring{$n \geq 5$}{n >= 5} and \texorpdfstring{$E_n$}{En}}
\label{subsec:proof-a}

We now deduce \Cref{thm:main-noncancellative} from \Cref{prop:D5-counterexample}.

\begin{theorem}\label{thm:noncancellative-D}
    Let $n \geq 5$ and let $w$ be a proper quasi-Coxeter element of a Coxeter group $W$ of type $D_n$.
    Then $M_w$ is neither left nor right cancellative, and the natural map $M_w \to G_w$ is not injective.
\end{theorem}

\begin{proof}
    We may assume that $w = w_{p,q}$ with $p \geq 3$ and $q \geq 2$.
    Let $x = [1,2,3]\,[p+1,p+2]$, which lies below $w$.
    Relabeling the coordinates $1, 2, 3, p+1, p+2$ as $1, \dots, 5$ identifies the interval $[1,x]$, together with the labels of its cover relations, with $[1,w_{3,2}] \cong \annularncD(3,2)$, so that $M_x \cong M_{w_{3,2}}$.
    By \Cref{prop:D5-counterexample}, the monoid $M_x$ is neither left nor right cancellative, and \Cref{prop:parabolic-submonoid} transfers this to $M_w$.
    The last statement follows as well.
\end{proof}

\begin{theorem}\label{thm:noncancellative-E}
    Let $w$ be a proper quasi-Coxeter element of a Coxeter group of type $E_6$, $E_7$, or $E_8$.
    Then $M_w$ is neither left nor right cancellative, and the natural map $M_w \to G_w$ is not injective.
\end{theorem}

\begin{proof}
    For each of the $14$ conjugacy classes of proper quasi-Coxeter elements, a computer search finds elements $x \le w$ of rank $5$ such that $|[1,x]| = 204$.
    Again by exhaustive search, one checks that this cardinality forces $(W_x, x)$ to be isomorphic to $(D_5, w_{3,2})$, and we conclude as in the proof of \Cref{thm:noncancellative-D}.
\end{proof}

\section*{Acknowledgments}
We thank Nathan Williams and Hugh Thomas for discussions on compatible reflection orders and EL-shellability, and for permission to include Thomas's argument in \Cref{rem:thomas.proof}.
We thank Christian Stump for helpful discussions and for communicating to us the closed formula for the type~$D$ zeta polynomial in \Cref{thm:zeta.polynomial.D}, which was suggested by AI.
We also thank Georges Neaime for early discussions on the combinatorics of quasi-Coxeter elements.

\section*{AI statement}
In the proof of \Cref{thm:compatible-EL-ABD}, AI helped find the chains of inequalities that exclude pairs of $w$-crossing reflections in types $B$ and $D$.
AI suggested the closed formula for the zeta polynomial of type~$D$ annular noncrossing partitions in \Cref{thm:zeta.polynomial.D}, which we subsequently proved.
AI was also used to generate code for the computations in the exceptional cases (\Cref{sec:exceptional-cases}) and for the search for noncancellativity examples (\Cref{sec:interval-monoid}).
The authors take full responsibility for the content of this paper.

\bibliographystyle{amsalpha-initials}
\bibliography{bibliography}

\end{document}